\documentclass[journal]{IEEEtran}

\ifCLASSINFOpdf
\else
\fi

\usepackage{times,epsf,amsmath,amssymb,cite}
\usepackage{caption2}
\usepackage{hyperref}
\usepackage{amsthm}
\usepackage{graphicx} 
\usepackage{mathabx}
\usepackage{makecell}
\usepackage{wrapfig}
\usepackage{bbm}
\usepackage{bm}

\newtheorem{Theorem}{Theorem}
\newtheorem*{theorem*}{Theorem}
\newtheorem{Proposition}{Proposition}
\newtheorem{Lemma}{Lemma}
\newtheorem{Corollary}{Corollary}
\newtheorem{Definition}{Definition}

\newcommand\norm[1]{\left\lVert#1\right\rVert}

\newcommand\bb[1]{\mathbf{#1}}

\newcommand{\R}{\mathbb{R}}

\newcommand{\C}{\mathcal{C}}
\newcommand{\M}{\mathcal{M}}

\newcommand{\K}{\mathrm{K}}
\newcommand{\BV}{\mathrm{BV}}

\newcommand{\D}{\mathcal{D}}
\newcommand{\Dd}{\mathrm{D}}
\newcommand{\DD}{\mathrm{D}\otimes\mathrm{D}}

\newcommand{\w}{\mathrm{weak}^{\star}}

\begin{document}

\title{Mixed-Derivative Total Variation}

\author{Vincent~Guillemet,
        and~Michael Unser}

% make the title area
\maketitle

% As a general rule, do not put math, special symbols or citations
% in the abstract or keywords.
\begin{abstract}
The formulation of norms on continuous-domain Banach spaces with exact pixel-based discretization is advantageous for solving inverse problems (IPs). In this paper, we investigate a new regularization that is a convex combination of a TV term and the $\M(\R^2)$ norm of mixed derivatives. We show that the extreme points of the corresponding unit ball are indicator functions of polygons whose edges are aligned with either the $x_1$- or $x_2$-axis. We then apply this result to construct a new regularization for IPs, which can be discretized exactly by tensor products of first-order B-splines, or equivalently, pixels. Furthermore, we exactly discretize the loss of the denoising problem on its canonical pixel basis and prove that it admits a unique solution, which is also a solution to the underlying continuous-domain IP.

\end{abstract}


\IEEEpeerreviewmaketitle

\section{Introduction}

The study of continuous-domain formulations of inverse problems (IP) is crucial in computational imaging \cite{bertero2021introduction,mccann2019biomedical}. In this setting, the objective is to recover an unknown image $f:\R^2\to\R$ from noisy measurements. Due to the pixel-based geometry of images, such problems are commonly addressed by seeking an approximate solution $f^{\star}$ in
\begin{equation}
    \mathcal{X}_1=\left\{f=\sum_{\bb{k}\in\mathbb{Z}^2}\bb{a}[\bb{k}]\mathbbm{1}_{E_{\bb{k}}}\right\},
\end{equation}
where the sets $(E_{\bb{k}})_{\bb{k}\in\mathbb{Z}^2}$ form a pixel-based decomposition of $\R^2$. In several acquisition modalities, the forward measurement process is not naturally associated with a canonical pixel basis, in which case $\mathcal{X}_1$ must be designed by the practitioner. In this context, it is advantageous to identify a continuous-domain Banach space $\mathcal{X}$ that contains all pixel-type basis functions and within which $\mathcal{X}_1$ can be refined. The search space $\mathcal{X}$ is typically selected jointly with the regularization functional $\text{Reg}(\cdot)$ \cite{gupta2018continuous,unser2019native}, which encodes prior knowledge about the geometry of $f^{\star}$ when the IP is ill-posed.

\subsection{TV Regularization}
A standard regularization for the pixel basis is the total variation (TV) \cite{rudin1992nonlinear}, which, for $f\in\mathcal{X}=\mathrm{BV}(\Omega)$ \cite{ambrosio2000functions,evans2018measure} and in its anisotropic form, is given by
\begin{multline}
\label{eq:1.A.2}
    \norm{\nabla\{f\}}_{\M(\R^2)^2} = \\
    \left(\norm{[\mathrm{I}\otimes\Dd]\{f\}}_{\M(\R^2)} + \norm{[\Dd\otimes\mathrm{I}]\{f\}}_{\M(\R^2)}\right),
\end{multline}
where $\norm{\cdot}_{\M(\R^2)}$ denotes the norm on the space of bounded measures \cite{rudin}.  
This norm has the advantage of being exactly discretizable on the pixel basis, which makes it particularly suitable for image processing. Specifically, when $f\in\mathcal{X}_1$ is a linear combination of square pixels with weights $\bb{a}$, \eqref{eq:1.A.2} can be discretized as  
\begin{equation}
\label{eq:1.A.3}
    c\left(\norm{\bm{h}_{0,1}\ast\bb{a}}_{\ell_1} + \norm{\bm{h}_{1,0}\ast\bb{a}}_{\ell_1}\right),
\end{equation}
where $c$ is determined by the pixel size, and  
\begin{equation}
    \bm{h}_{0,1} = \begin{bmatrix}
        -1\\
        1
    \end{bmatrix}, \quad
    \bm{h}_{1,0} = \begin{bmatrix}
        -1 & 1
    \end{bmatrix}.
\end{equation}
The search for the optimal $f^{\star} \in \mathcal{X}_1$ thus reduces to a compressed sensing problem \cite{donoho2006compressed, candes2007sparsity, bruckstein2009sparse}, which can be solved using classical algorithms such as FISTA \cite{beck2009fast} or primal-dual methods \cite{zhu2008efficient,esser2010general,chambolle2011first,condat2013primal}. One reason for the popularity of TV regularization is the strong theoretical foundation on which it is built.  

The norm $\norm{\nabla\{\cdot\}}_{\M(\R^2)^2}$ is not only well defined for pixels, but also for indicator functions of Caccioppoli sets \cite{caccioppoli1952misura,de1954teoria}, i.e., sets that can be approximated in measure by polygons.  
Within the framework of representer theorems (RTs) for Banach spaces \cite{flinth2019exact,unser2017splines,unser2021unifying,boyer2019representer}, Bredies and Carioni \cite{bredies2020sparsity} built upon the results of \cite{fleming1957functions,fleming1960functions,ambrosio2001connected} to show that TV regularization promotes solutions to inverse problems that are finite linear combinations of simple sets,  Caccioppoli sets whose boundaries are represented by a unique Jordan curve \cite{ambrosio2001connected}.  

In contrast to the spline solutions promoted by the one-dimensional TV counterpart $\norm{\Dd\{\cdot\}}_{\M(\R)}$, the generality of these simple sets prevents them from benefiting from the convenient spline structure \cite{unser1999splines,debarre2019b,debarre2022uniqueness,guillemet2025convergence,fageot2025variational}.

\subsection{Mixed-Derivative Regularization}
One approach to promoting pixel-type solutions with a spline structure is to impose additional regularity. In \cite{guillemet2025tensor}, the authors considered the search space $\M_{\DD}(\K)$, consisting of functions $f$ whose generalized gradient $\nabla\{f\}$ and generalized mixed derivative $[\DD]\{f\}$ are bounded measures supported in $\K = [0,1]\times[0,1]$. They showed that the extreme points of the unit ball associated with the norm $\norm{[\DD]\{\cdot\}}_{\M(\R^2)}$ are tensor products of $\Dd$-splines. However, this formulation is not entirely satisfactory because the regularization $\norm{[\DD]\{\cdot\}}_{\M(\R^2)}$ does not control the perimeter of sets, but only their number of corners. As a result, it may promote solutions containing components that are undesirable, such as indicators of unbounded sets.

\subsection{Contributions}
In this paper, for $\theta \in ]0,1]$, we regularize with the new norm  
\begin{equation}
    \norm{\cdot}_{\theta} = \theta \norm{[\DD]\{\cdot\}}_{\M(\R^2)} + (1-\theta) \norm{\nabla\{\cdot\}}_{\M(\R^2)^2},
\end{equation}
which generalizes the anisotropic TV. Our contributions are fourfold.

\textbf{Functional Framework:} We define the search space $\M_{\DD,0}(\K)$ as the set of functions $f \in \M_{\DD}(\K)$ supported in $\K$. We prove that this space is a dual Banach space for the norm $\norm{[\DD]\{\cdot\}}_{\M}$, and we establish that all $\norm{\cdot}_{\theta}$-norms are equivalent.  

\textbf{Extreme Point Characterization:} We show that the extreme points of the unit ball of the convex cone of positive functions $\left(\M_{\DD,0}^+(\K), \norm{\cdot}_{\theta}\right)$ are indicator functions of polygons whose edges are parallel to either the $x_1$- or $x_2$-axis. We then apply this result to the regularization of abstract IPs with finite data.

\textbf{Resolution of IPs:} We demonstrate that tensor-product B-splines of order 1 exactly discretize a broad class of IP loss functionals, and that $\norm{\cdot}_{\theta}$ can be discretized exactly as  
\begin{equation}
    \theta c \norm{\bm{h}_{0,1} \ast \bb{a}}_{\ell_1} +
    \theta c \norm{\bm{h}_{1,0} \ast \bb{a}}_{\ell_1} +
    (1-\theta) \norm{\bm{h}_{1,1} \ast \bb{a}}_{\ell_1},
\end{equation}
where $\bm{h}_{1,1} = \bm{h}_{0,1} \otimes \bm{h}_{1,0}$. We further prove that these exact discretizations converge, in an appropriate sense, to the continuous-domain IP as the pixel size tends to zero.

\textbf{Application to Denoising:} We prove that the canonical pixel-based discretization of the denoising loss admits a unique solution, which also solves the corresponding continuous-domain problem. Finally, we numerically assess the efficiency of the proposed method.

\subsection{Related Works}
\textbf{Higher-Order Generalization of TV.}  
Bredies et al.\ introduced the generalized total variation (GTV) \cite{bredies2010total,bredies2014regularization,bredies2020sparsity}, which incorporates all derivatives up to a specified order $k$. In particular, the second-order TV ($k=2$) is defined on the space of functions with bounded Hessian \cite{demengel1984fonctions} and, for a smooth function $f$, takes the form  
\begin{equation}
    TV^{(2)}\{f\} = \int_{\R^2} \norm{\nabla^{2}\{f\}(\bb{x})}_F \,\mathrm{d}\bb{x},
\end{equation}
where $\norm{\cdot}_F$ denotes the Frobenius norm. Lefkimiatis et al.\ extended this notion by introducing the Hessian TV (HTV), in which the Frobenius norm is replaced by a Schatten-$p$ norm \cite{lefkimmiatis2011hessian,lefkimmiatis2013hessian,aziznejad2023measuring}. Regularization with HTV has been employed for the learning of continuous piecewise-linear (CPWL) mappings \cite{9655475,pourya2023delaunay,pourya2024box}. Ambrosio et al.\ showed that the closure, in the unit ball of the space of functions with finite HTV, of the set of CPWL extreme points contains all extremal points \cite{ambrosio2023functions,ambrosio2024linear}.  

Our regularization lies, in terms of regularity constraints, between TV and HTV. Specifically, it enforces more regularity than TV due to the presence of the $\DD$ term, but less than HTV because it omits the terms $\mathrm{I} \otimes \Dd^2$ and $\Dd^2 \otimes \mathrm{I}$. In particular, HTV is not well defined on pixel basis.

\hfill\\
\textbf{Extreme Points of Mixed Norms.}  
We investigate the extreme points of the unit ball associated with our regularization norm, which takes the form  
\[
\norm{\cdot} = \theta \norm{\cdot}_A + (1-\theta) \norm{\cdot}_B,
\]
where $\norm{\cdot}_A$ and $\norm{\cdot}_B$ are two norms whose extreme points are known. When the search space $\mathcal{X}$ can be decomposed as a direct sum $\mathcal{X}_A \oplus \mathcal{X}_B$, with $\norm{\cdot}_A\vert_{\mathcal{X}_B} = 0$ and $\norm{\cdot}_B\vert_{\mathcal{X}_A} = 0$, the extreme points of $\norm{\cdot}$ have been systematically characterized \cite{unser2022convex}. However, when such a direct-sum decomposition is not possible, as is the case in this work, the description of the extreme points becomes significantly more challenging. In finite dimensions, they are located at the intersections of the hyperplanes that define the boundaries of the unit balls \cite{dubins1962extreme,parhi2023sparsity,unser2025universal}. To the best of our knowledge, no analogous characterization is currently available in infinite-dimensional settings.

\hfill\\
\textbf{Continuous Data.}  
TV was originally introduced as a regularization for the denoising loss $\norm{f - \cdot}_{\mathcal{L}_2(\Omega)}$, where $\Omega$ is a domain and $f$ denotes the continuous data \cite{rudin1992nonlinear}. Subsequent works have deepened the theoretical understanding of this approach. The existence and uniqueness of solutions for this problem, as well as for the TV flow problem, were established in \cite{bellettini2002total,andreu2001minimizing,andreu2002some}. The regularity of solutions was analysed in \cite{caselles2011regularity}, while discretization and convergence properties, including convergence rates, were investigated in \cite{lai2009convergence,bartels2012total,wang2011error}. Chambolle and Pock provide a comprehensive overview of TV discretization and convergence techniques in \cite{chambolle2021approximating}. In the present work, we focus on the study of IPs with finite data.

\section{Theoretical Formulation}
\subsection{Mixed-Derivative Native Space}
Let $\K = [0,1]^2$ denote the unit square. We define the native space $\M_{\DD,0}(\K)$ as  
\begin{equation}
\label{eq:2.A.8}
\left\{ f \in \D'(\R^2) :\quad
\begin{cases}
    [\DD]\{f\} \in \M(\K), \\
    \forall \phi \in \C_c^{\infty}(\K^c),\ \langle f, \phi \rangle = 0
\end{cases}
\right\},
\end{equation}
where $\M(\K)$ denotes the space of bounded measures supported in $\K$, and $\K^c = \R^2 \setminus \K$.  
The second condition in \eqref{eq:2.A.8} enforces that the function $f$ vanishes outside $\K$. In Proposition~\ref{prop:carac}, we provide a more practical, integral-based characterization of the functions $f$ belonging to $\M_{\DD,0}(\K)$. To this end, we introduce the integral transform $[\DD]^{-1}\{\cdot\}$, defined for all $m \in \M(\K)$ by  
\begin{align}
    [\DD]^{-1}\{m\} &= \int_{\K} u(\cdot - x_1)\, u(\cdot - x_2) \,\mathrm{d}m(x_1,x_2)\nonumber \\
    &= (u \otimes u) \ast m,
\end{align}
where $u$ is the heaviside function.
\begin{Proposition}
\label{prop:carac}
A function $f$ belongs to $\M_{\DD,0}(\K)$ if and only if it belongs to  
\begin{equation}
\label{eq:carac}
    \left\{ f \in [\DD]^{-1}\{\M(\K)\} : 
    \forall \phi \in \C_c^{\infty}(\K^c),\ \langle f, \phi \rangle = 0
    \right\}.
\end{equation}
\end{Proposition}

\begin{proof}[\textbf{Proof of Proposition \ref{prop:carac}}]
Clearly, if $f$ satisfies \eqref{eq:carac}, then $f \in \M_{\DD,0}(\K)$. To prove the converse, assume $f \in \M_{\DD,0}(\K)$ and define $m = [\DD]\{f\}$. We claim that 
\begin{equation}
    [\DD]^{-1}\{m\} = f
\end{equation}
which proves the proposition. Observe that $[\DD]^{-1}\{m\}$ may differ from $f$ only by an element of the null space of $\DD$ in $\D'(\R^2)$. Hence, one can write  
\begin{equation}
    f = [\DD]^{-1}\{m\} + \phi_1 \otimes 1 + 1 \otimes \phi_2,
\end{equation}
for some $(\phi_1, \phi_2) \in \D'(\R)^2$. Note that  
\begin{align}
    \text{supp}(m) \subset \K \Rightarrow & \text{supp}\left( [\DD]^{-1}\{m\} \right) \subset [0, \infty[^2 \nonumber\\
    \Rightarrow & \text{supp}(\phi_1 \otimes 1 + 1 \otimes \phi_2) \subset [0, \infty[^2,
\end{align}
which directly implies that $\phi_1 \otimes 1 + 1 \otimes \phi_2 = 0$ and proves the claim.
\end{proof}

Let $\theta \in ]0,1]$, we equip $\M_{\DD,0}(\K)$ with the mixed norm $\norm{\cdot}_{\theta}$ defined as
\begin{equation}
    \norm{\cdot}_{\theta} = \theta \norm{[\DD]\{\cdot\}}_{\M(\K)} + (1-\theta) \norm{\nabla\{\cdot\}}_{\M(\K)^2},
\end{equation}
with 
\begin{equation}
    \norm{\nabla\{\cdot\}}_{\M(\K)^2} = \norm{[\mathrm{I} \otimes \Dd]\{\cdot\}}_{\M(\K)} + \norm{[\Dd \otimes \mathrm{I}]\{\cdot\}}_{\M(\K)}.
\end{equation}
In the sequel, we write $\norm{\cdot}_{\M}$ instead of $\norm{\cdot}_{\M(\K)}$ and $\norm{\cdot}_{\M^2}$ instead of $\norm{\cdot}_{\M(\K)^2}$. The key properties of $\left(\M_{\DD,0}(\K), \norm{\cdot}_{\theta}\right)$ are summarised in Theorem \ref{th:normequivalence}.

\begin{Theorem}
\label{th:normequivalence}
    Let $\theta \in ]0,1]$. Then, the vector space $\big(\M_{\DD,0}(\K), \norm{\cdot}_{\theta}\big)$ satisfies the following properties:
    \begin{itemize}
        \item[1.] All norms $\norm{\cdot}_{\theta}$ are well-defined and equivalent; that is, for all $(\theta,\theta') \in ]0,1]^2$, there exists a constant $C > 0$ such that 
        \begin{equation}
        \label{eq:normequivalence} 
            \forall f \in \M_{\DD,0}(\K): \quad \norm{f}_{\theta} \leq C \norm{f}_{\theta'}.
        \end{equation}
        \item[2.] It is a Banach space.
        \item[3.] Equipped with the norm $\norm{[\DD]\{\cdot\}}_{\M}$, it is the dual space of $[\DD]\{\C_0(\R^2)\} / M_0$, where
        \begin{multline}
            M_0 = \big\{ v \in [\DD]\{\C_0(\R^2)\} : \\
            \forall f \in \M_{\DD,0}(\K), \langle f, v \rangle = 0 \big\},
        \end{multline}
        and $[\DD]\{\C_0(\R^2)\}$ is equipped with the norm
        \begin{equation}
            \norm{[\DD]\{\phi\}} = \norm{\phi}_{\infty(\R^2)}.
        \end{equation}
    \end{itemize}
\end{Theorem}

\begin{proof}[\textbf{Proof of Theorem \ref{th:normequivalence}}]
\hfill\\
\textbf{Item 1.}
Let $m=[\DD]\{f\}$ and $f=(u\otimes u)\ast m$. Then,  
\begin{align}
    &\norm{[\mathrm{I}\otimes\Dd]\{f\}}_{\M}=\norm{(u\otimes1)\ast m}_{\M}\leq\vert\K\vert\norm{m}_{\M}\nonumber\\
    \Rightarrow&\norm{\nabla\{f\}}_{\M}\leq2\vert\K\vert\norm{[\DD]\{f\}}_{\M}
\end{align}
where $\vert\K\vert=\int_{\K}1\mathrm{d}\bb{x}.$ Consequently, \eqref{eq:normequivalence} holds with
\begin{align}
     C=\text{max}\left(\frac{\theta}{\theta'},\frac{1-\theta}{1-\theta'}\right),&\quad\text{for}\quad\theta'\in]0,1[\\
    C=(\theta+(1-\theta)\vert\K\vert2),&\quad\text{for}\quad\theta'=1.
\end{align}
\textbf{Item 2.} We know from \cite{guillemet2025tensor} that $[\DD]^{-1}\{\M(\K)\}$ is a Banach space for the norm $\norm{[\DD]\{\cdot\}}_{\M }$. Since $\M_{\DD,0}(\K)$ is a closed subspace of $[\DD]^{-1}(\M(\K))$, it is also Banach space for $\norm{[\DD]\{\cdot\}}_{\M }$. By the norm equivalence, it follows that it is a Banach space for $\norm{\cdot}_{\theta}$.
\hfill\\
\textbf{Item 3.} Here, all quotient spaces are equipped with their quotient norms. We know from \cite{guillemet2025tensor} that $[\DD]^{-1}\{\M(\K)\}$ equipped with $\norm{[\DD]\{\cdot\}}_{\M }$ is the dual space of $[\DD]\{\C_0(\R^2)\}/M$, with
\begin{multline}
    M=\{v\in[\DD]\{\C_0\}(\R^2):\\
    \forall f\in[\DD]^{-1}\{\M(\K)\},\langle f,v\rangle=0\}.
\end{multline}
Since $\M_{\DD,0}(\K)$ is a Banach subspace of $[\DD]^{-1}\{\M(\K)\}$, the predual of $\M_{\DD,0}(\K)$ is given by the (pre-)annihilator 
\begin{equation}
    \left([\DD]\{\C_0(\R^2)\}/M\right)/M_0=[\DD]\{\C_0(\R^2)\}/M_0.
\end{equation}
\end{proof}

The key takeaway of Theorem \ref{th:normequivalence} is that all $\theta$-norms, for $\theta>0$, are topologically equivalent on $\M_{\DD,0}(\K)$. We have shown that $\M_{\DD,0}(\K)$ also admits a $\w$ topology, such that $(f_n)_{n=1}^{\infty}$ is $\w$-convergent to $f$ in $\M_{\DD,0}(\K)$ if and only if
\begin{equation}
    \forall [v]\in[\DD]\{\C_0(\R^2)\}/M_0:\quad\lim_{n\to\infty}\langle f_n,[v]\rangle=\langle f,[v]\rangle.
\end{equation}
Observe that, even though the quotient seems to complicate the predual structure, one still has that $(f_n)_{n=1}^{\infty}$ is $\w$-convergent to $f$ if and only if
\begin{equation}
    \forall v\in[\DD]\{\C_0(\R^2)\}:\quad\lim_{n\to\infty}\langle f_n,v\rangle=\langle f,v\rangle,
\end{equation}
because changing the representative in the quotient class, by definition, does not change the value of the brackets.

Next, we discuss the relation of $\M_{\DD,0}(\K)$ with the space $\BV(\Omega)$ of functions of bounded variation, which is naturally equipped with the norm $\norm{\nabla\{\cdot\}}_{\M^2}$. We recall that, for a domain $\Omega$,
\begin{equation}
    \BV(\Omega)=\{f\in\mathcal{L}_1(\Omega):\nabla\{f\}\in\M(\Omega)^2\}.
\end{equation}
\begin{Proposition} 
\label{prop:BV}
For a domain $\Omega$ such that $\K\subset\Omega$, the following holds:
\begin{itemize}
    \item [1.] the functional $\norm{\nabla\{\cdot\}}_{\M^2}$ is a norm on $\M_{\DD,0}(\K)$;
    \item [2.] the space $\M_{\DD,0}(\K)$ is a subspace of $\BV(\Omega)$, i.e.
    \begin{equation}
        \M_{\DD,0}(\K)\subset\BV(\Omega);
    \end{equation}
    \item [3.] the vector space $\big(\M_{\DD,0}(\K),\norm{\nabla\{\cdot\}}_{\M^2}\big)$ is not closed.
\end{itemize}
\end{Proposition}

\begin{proof}[\textbf{Proof of Proposition \ref{prop:BV}}]
\hfill\\
\textbf{Item 1.} Absolute homogeneity, the triangle inequality, and non-negativity are straightforward. In addition, 
\begin{equation}
    \norm{\nabla\{f\}}_{\M^2}=0 \Leftrightarrow f=c
\end{equation}
for some $c \in \mathbb{R}$. Since $f$ is compactly supported, one has that $c=0$.
\hfill\\
\textbf{Item 2.} 
It follows from Item 1 of Theorem \ref{th:normequivalence} that $\nabla\{f\} \in \M(\mathbb{R}^2)^2$. Because $\mathrm{supp}(\nabla\{f\}) \subset \K \subset \Omega$, we find that $\nabla\{f\} \in \M(\Omega)^2$. The same argument applies to conclude that $f \in \mathcal{L}_1(\Omega)$.
\hfill\\
\textbf{Item 3.}
Let $\K=[0,2]^2$ and $\beta(\cdot) = \mathbbm{1}_{[0,1]}(\cdot)$. We define
\begin{equation}
    f_{n+1}(\cdot) = \sum_{k=0}^{n} \beta\left(\frac{\cdot}{2^{-k}}\right) \otimes \beta\left(\frac{\cdot - c_k}{2^{-k}}\right)
\end{equation}
with $c_k = 2 \left(1 - \frac{1}{2^{k}} \right)$. Then, a straightforward calculation shows that the sequence $(f_n)_{n=1}^\infty$ with $f_n \in \M_{\DD,0}(\K)$ converges in $\norm{\nabla\{\cdot\}}_{\M^2}$ to $f^\infty$, which is a staircase with an accumulation of steps towards the $x_2$ axis. In addition, we calculate that 
\begin{multline}
    [\DD]\{f^\infty\} = \\
    \delta_{(0,0)} - \delta_{(1,0)} + \sum_{k=1}^{\infty} \delta_{(2 - c_k, c_k)} - \delta_{(2 - c_{k+1}, c_k)},
\end{multline}
which is an unbounded measure with $\norm{[\DD]\{f^\infty\}}_{\M} = \infty$. Consequently, $\big(\M_{\DD,0}(\K), \norm{\nabla\{\cdot\}}_{\M^2}\big)$ is not closed.
\end{proof}

\subsection{Sets, Indecomposable Sets and Approximations}
We study the geometrical properties of sets that satisfy $\norm{[\DD]\{\mathbbm{1}_A\}}_{\M} < \infty$. As a first step, we show that an indicator function $\mathbbm{1}_A$, as an element of $\M_{\DD,0}(\K)$, forms an equivalence class.

\begin{Proposition}
\label{prop:setdiff}
    Let $A \subset \K$ be a measurable set. If a measurable set $A' \subset \K$ differs from $A$ by a set of measure zero, then 
    \begin{equation}
        \forall \theta \in ]0,1], \quad \norm{\mathbbm{1}_A - \mathbbm{1}_{A'}}_{\theta} = 0.
    \end{equation}
    In particular, if $\mathbbm{1}_A \in \M_{\DD,0}(\K)$, then $\mathbbm{1}_{A'} \in \M_{\DD,0}(\K)$.
\end{Proposition}

\begin{proof}[\textbf{Proof of Proposition \ref{prop:setdiff}}]
We calculate that  
\begin{align}
&\forall \phi \in \C_c^{\infty}(\R^2), \quad \langle \mathbbm{1}_A - \mathbbm{1}_{A'}, [\DD]\{\phi\} \rangle = 0 \label{eq:2.B.36}\\
\Rightarrow & \forall \phi \in \C_0(\R^2), \quad \langle \mathbbm{1}_A - \mathbbm{1}_{A'}, [\DD]\{\phi\} \rangle = 0 \label{eq:2.B.37}\\
\Rightarrow & \norm{\mathbbm{1}_A - \mathbbm{1}_{A'}}_1 = 0 \label{eq:2.B.38}\\
\Rightarrow & \norm{\mathbbm{1}_A - \mathbbm{1}_{A'}}_{\theta} = 0. \label{eq:2.B.39}
\end{align}
In \eqref{eq:2.B.36}, we used the assumption that $A'$ differs from $A$ by a set of measure zero. In \eqref{eq:2.B.37} we used the fact that $\C_c^{\infty}(\R^2)$ is dense in $\C_0(\R^2)$, and that the linear and continuous functional $\langle \mathbbm{1}_A - \mathbbm{1}_{A'}, \cdot \rangle$ admits a unique (linear and continuous) extension to $\C_0(\R^2)$. In \eqref{eq:2.B.38} we used Item 3 of Theorem \ref{th:normequivalence}. In \eqref{eq:2.B.39} we used Item 1 of Theorem \ref{th:normequivalence}. Finally,
\begin{equation}
    \mathbbm{1}_{A'} = \mathbbm{1}_A - (\mathbbm{1}_A - \mathbbm{1}_{A'}) \in \M_{\DD,0}(\K).
\end{equation}
\end{proof}

The fact that $\mathbbm{1}_A$ forms an equivalence class is often forgotten in the sequel as it is (often) clear from the context which representative to take. Proposition \ref{prop:finitecorners} shows that such indicators must have the form of a polygon with finitely many corners and edges parallel to the $x_1$ or the $x_2$ axis. 

\begin{Proposition}
\label{prop:finitecorners}
    If $\mathbbm{1}_A \in \M_{\DD,0}(\K)$, then $[\DD]\{\mathbbm{1}_A\}$ is a discrete measure with finitely many atoms. In addition, one has the representation
    \begin{equation}
    \label{eq:}
        \mathbbm{1}_A = \sum_{k=1}^K a_k u(\cdot - x_{1,k}) u(\cdot - x_{2,k}), \quad \text{with} \quad a_k \in \{-1,1\}
    \end{equation}
and $(x_{1,k}, x_{2,k}) \in \K.$
\end{Proposition}

\begin{proof}[\textbf{Proof of Proposition \ref{prop:finitecorners}}]
\hfill\\\textbf{Step 1.}
We define
\begin{align}
    \mathbbm{1}_A(\bb{t})&=\int_{\K}u(t_1-x_1)u(t_2-x_2)\mathrm{d}m(x_1,x_2)\nonumber\\
    &=m([0,t_1]\times[0,t_2])\label{eq:2.B.59}.
\end{align}
Observe that in \eqref{eq:2.B.59} the value of $\mathbbm{1}_A(\bb{t})$ is set by measuring the closed square. Other choices, such as $m([0,t_1[\times[0,t_2[)$, would yield functions that fall in the same equivalence class as $\mathbbm{1}_A.$
\hfill\\\textbf{Step 2.}
We assume by contradiction that there exists a set $E$ of the form
\begin{align}
    \label{eq:2.B.43}
    E=]e_{1}^-,e_1^+]\times]e_{2}^-,e_2^+]\cap\K,\quad\text{with}\quad(e_1^+,e_2^+)\in\K,
\end{align} 
such that $m(E)\notin\mathbb{Z}.$ Observe that, for  $E_{0,0}=[0,e_1^-]\times[0,e_2^-]\cap\K,$ $E_{1,0}=]e_1^-,e_1^+]\times[0,e_2^-]\cap\K,$ and $E_{0,1}=[0,e_1^-]\times]e_{2}^-,e_2^+]\cap\K$, one has that 
\begin{align}
    m(E+E_{0,0}+E_{1,0}+E_{0,1})=\mathbbm{1}_A(e_1^+,e_2^+)\in\{0,1\}.
\end{align}
Consequently, we either have that 
\begin{itemize}
    \item $m(E_{0,0})\notin\mathbb{Z},$ in which case $m(E_{0,0})=\mathbbm{1}_A(e_1^-,e_2^-)\in\{0,1\}$, which is a contradiction;
    \item  $m(E_{0,1})\notin\mathbb{Z},$ in which case 
    \begin{itemize}
        \item if $m(E_{0,0})\notin\mathbb{Z}$, then we find a contradiction as before;
        \item if $m(E_{0,0})\in\mathbb{Z}$, then $\mathbb{Z}\notni m(E_{0,1}+E_{0,0})=\mathbbm{1}_A(e_1^-,e_2^+)\in\{0,1\}$
        which is a contradiction;
    \end{itemize}
    \item $m(E_{1,0})\notin\mathbb{Z},$ in which case we reach a contradiction as in the $m(E_{0,1})\notin\mathbb{Z}$ case.
\end{itemize}
 Further, one can prove that the measure $m$
\begin{itemize}
    \item [1.] is integer valued on intersections of sets of the form \eqref{eq:2.B.43};
    \item [2.] is integer valued on unions of sets of the form \eqref{eq:2.B.43};
    \item [3.] is integer valued on the relative complement of a set of the form \eqref{eq:2.B.43}.
\end{itemize}
It follows that $m$ is integer-valued
on any Borel set \cite[Proposition 1.1.5]{cohn2013measure}. This implies, with the boundedness of $m$, that \cite[Chapters 1-2]{kallenberg2017random}
\begin{equation}
    m=\sum_{k=1}^K a_k \delta_{(x_{1,k}, x_{2,k})}, \quad a_k \in \mathbb{Z} \setminus \{0\}.
\end{equation}
The fact that $\mathbbm{1}_A$ is $\{0,1\}$-valued finally implies that $a_k \in \{-1,1\}$. 
\end{proof}

One important class is formed by indecomposable sets \cite{bredies2020sparsity,ambrosio2001connected}, which we relabel as $\nabla$-indecomposable in Definition \ref{def:decomposable}. These sets are central in our description of the unit ball of $(\M_{\DD,0}^+(\K),\norm{\cdot}_{\theta})$, as the indecomposability property translates into an extreme point structure. We extend this class to $[\DD]$-indecomposable sets in Definition \ref{def:decomposable2}.

\begin{Definition}
\label{def:decomposable}Let $A\subset\K$ be such that $\mathbbm{1}_A\in\M_{\DD,0}(\K)$. The set $A$ is $\nabla$-\emph{decomposable} if there exists a partition $A_1,A_2$ of $A$ such that, $\mathbbm{1}_{A_1} \in \M_{\DD,0}(\K)$, $\mathbbm{1}_{A_2} \in \M_{\DD,0}(\K)$,  
\begin{equation}
    \norm{\nabla\{\mathbbm{1}_{A_1}\}}_{\M^2}+\norm{\nabla\{\mathbbm{1}_{A_2}\}}_{\M^2}=\norm{\nabla\{\mathbbm{1}_{A}\}}_{\M^2}<\infty.
\end{equation}
and $\vert A_1\vert>0,$ $\vert A_2\vert>0.$ If the set $A$ is not $\nabla$-decomposable, it is $\nabla$-\emph{indecomposable}. The set of all $\nabla$-indecomposable sets is denoted by 
\begin{equation}
    \mathcal{E}_{\nabla}.
\end{equation}
\end{Definition}

\begin{Definition}
\label{def:decomposable2}
  Let $A \subset \K$ be such that $\mathbbm{1}_A \in \M_{\DD,0}(\K)$. The set $A$ is $[\DD]$-\emph{decomposable} if there exists a partition $A_1, A_2$ of $A$ such that, $\mathbbm{1}_{A_1} \in \M_{\DD,0}(\K)$, $\mathbbm{1}_{A_2} \in \M_{\DD,0}(\K)$,  
\begin{multline}
\label{eq:2.B.32}
     \norm{[\DD]\{\mathbbm{1}_{A_1}\}}_{\M} + \norm{[\DD]\{\mathbbm{1}_{A_2}\}}_{\M} = \\
        \norm{[\DD]\{\mathbbm{1}_{A}\}}_{\M} < \infty,
\end{multline}
and $\vert A_1\vert>0,$ $\vert A_2\vert>0.$ If the set $A$ is not $[\DD]$-decomposable, it is $[\DD]$-\emph{indecomposable}. The set of all $[\DD]$-indecomposable sets is denoted by 
\begin{equation}
    \mathcal{E}_{\DD}.
\end{equation}
\end{Definition}

The relation between $[\DD]$-indecomposability and $\nabla$-indecomposability is revealed in Proposition \ref{prop:sets}. 

\begin{Proposition}
\label{prop:sets}
The strict inclusion holds true
\begin{equation}
    \mathcal{E}_{\DD}\subset\mathcal{E}_{\nabla}.
\end{equation}
\end{Proposition}

 \begin{proof}[\textbf{Proof of Proposition \ref{prop:sets}}]
\hfill\\
\textbf{Item 1.} We provide an example of a $\nabla$-indecomposable set $A$ that is not $[\DD]$-indecomposable. To do so, we define 
\begin{align}
   & \mathbbm{1}_{A_1}(t_1,t_2)=\beta\left(\frac{t_1}{3}\right)\otimes\beta(t_2)\\
    & \mathbbm{1}_{A_2}(t_1,t_2)=\beta(t_1-1)\otimes\beta(t_2-1)
\end{align}
as well as $\mathbbm{1}_{A}=\mathbbm{1}_{A_1}+\mathbbm{1}_{A_2}$. On the one hand, $A$ is $[\DD]$-decomposable as, $\forall i\in\{1,2\}$,
\begin{equation}
    \norm{[\DD]\{\mathbbm{1}_{A_i}\}}_{\M}=4,\quad\norm{[\DD]\{\mathbbm{1}_{A}\}}_{\M}=8.
\end{equation}
On the other hand, we know from \cite[Proposition 2]{ambrosio2001connected} that $A$ is $\nabla$-indecomposable.

\hfill\\
\textbf{Item 2.} We show that, if $(A_1, A_2)$ is a $\nabla$-decomposition of $A$, then it is also a $[\DD]$-decomposition. We define 
$m = [\DD]\{\mathbbm{1}_A\}$, $m_1 = [\DD]\{\mathbbm{1}_{A_1}\}$, and $m_2 = [\DD]\{\mathbbm{1}_{A_2}\}$. It follows from Proposition \ref{prop:finitecorners} that $m$, $m_1$, and $m_2$ are discrete measures with finitely many atoms, such that
\begin{equation}
m_i = \sum_{k=1}^{K_i} a_{i,k} \delta_{b_{i,k}} \otimes \delta_{c_{i,k}} \quad \text{with} \quad a_{i,k} \in \{-1,1\}.
\end{equation}
Assume, by contradiction, that $(A_1, A_2)$ fails to be a $[\DD]$-decomposition. Then, there exists $(k,k')$ such that 
\begin{equation}
    (b_{1,k}, c_{1,k}) = (b_{2,k'}, c_{2,k'}), \quad \text{and} \quad a_{1,k} = -a_{2,k'}.
\end{equation}
This means that $(b_{1,k}, c_{1,k})$ (respectively $(b_{2,k'}, c_{2,k'})$) is a corner of $A_1$ (respectively $A_2$). Observe that, for $(A_1, A_2)$ to be a $\nabla$-decomposition of $A$, their edges may overlap only on corners, and not on segments. This implies that both corners are facing each other, in which case $a_{1,k}$ and $a_{2,k'}$ must have the same value. This is a contradiction.

\end{proof}

Indicator functions $\mathbbm{1}_A \in \M_{\DD,0}(\K)$ are not only useful for the indecomposable structure, but also for their approximation properties. We conclude this section with Lemma \ref{lemma:approximation}, which states that any function $f \in \M_{\DD,0}(\K)$ can be approximated, in an appropriate sense, by linear combinations of indicator functions. The proof is given in Appendix \ref{app:proofs}.

\begin{Lemma}
\label{lemma:approximation}
    For $f\in\M_{\DD,0}(\K)$, there exists a sequence $(f_n)_{n=1}^{\infty}$ of functions $f_n\in\M_{\DD,0}(\K)$, such that $[\DD]\{f_n\}$ is a discrete measure supported on $\frac{1}{n}[1\cdots n]\times\frac{1}{n}[1\cdots n],$ and such that
    \begin{itemize}
        \item [1.] it is $\w$-convergent to $f$ in $\M_{\DD,0}(\K)$;
        \item [2.] over the mixed derivative, $$\norm{[\DD]\{f\}}_{\M}\geq\norm{[\DD]\{f_n\}}_{\M}$$ and $$\norm{[\DD]\{f\}}_{\M}=\underset{n\to\infty}{\mathrm{lim}}\norm{[\DD]\{f_n\}}_{\M};$$
        \item [3.] it is convergent to $f$ in $\mathcal{L}_1(\K)$;
        \item [4.] over the gradient, $\norm{\nabla\{f\}}_{\M^2}=\underset{n\to\infty}{\mathrm{lim}}\norm{\nabla\{f_n\}}_{\M^2}$.
    \end{itemize}
\end{Lemma}

\subsection{Coarea and Cocorner Formulas}
In the sequel, we will use the (convex) cone $\M_{\DD,0}^+(\K)$ of positive functions, defined as 
\begin{equation}
\label{eq:2.A.1}
    \big\{f \in \M_{\DD,0}(\K) : \forall \phi \in \C_c^{\infty}(\R^2), \ \phi \geq 0 \Rightarrow \langle f, \phi \rangle \geq 0 \big\},
\end{equation}
which is closed in both the strong and the $\w$ topologies.

We develop tools to split a function $f \in \M_{\DD,0}^+(\R^2)$ into a sum of indicator functions on indecomposable sets. To this end, we first recall the coarea formula. For $f \in \M_{\DD,0}^+(\R^2)$ and $s \in [0, \infty[$, we define the functions
\begin{equation}
    P_-(f,s) = \norm{\nabla\{\min(f,s)\}}_{\M^2}
\end{equation}
and
\begin{equation}
    P_+(f,s) = \norm{\nabla\{\max(f - s, 0)\}}_{\M^2}.
\end{equation}

\begin{Theorem}
\label{th:coarea}
The functions $P_-(f,\cdot)$ and $P_+(f,\cdot)$ are continuous, respectively increasing and decreasing, and satisfy the equalities
\begin{equation}
\label{eq:2.C.51}
    P_-(f,s) = \int_{0}^s \norm{\nabla\{\mathbbm{1}_{\{x : f(x) \geq \alpha\}}\}}_{\M^2} \, \mathrm{d}\alpha,
\end{equation}
\begin{equation}
\label{eq:2.C.52}
    P_+(f,s) = \int_{s}^{\infty} \norm{\nabla\{\mathbbm{1}_{\{x : f(x) \geq \alpha\}}\}}_{\M^2} \, \mathrm{d}\alpha.
\end{equation}
In addition, they satisfy the $BV$-coarea formula 
\begin{multline}
\label{eq:2.C.53}
    \norm{\nabla\{f\}}_{\M^2} = P_-(f,s) + P_+(f,s) \\
    = \int_{0}^{\infty} \norm{\nabla\{\mathbbm{1}_{\{x : f(x) \geq \alpha\}}\}}_{\M^2} \, \mathrm{d}\alpha.
\end{multline}
\end{Theorem}

The coarea formula, for the $\ell_2$ norm on the gradient, was proved for Lipschitz functions in \cite{federer1959curvature}, for BV functions in \cite{fleming1960integral} and was later extended for general norms (for example, see \cite{rotem2023anisotropic}). Equations \eqref{eq:2.C.51} and \eqref{eq:2.C.52} follow from the application of the coarea formula on $\text{min}(f,s)$ and $\text{max}(f-s,0)$. 

Next, we provide a new formula inspired by the coarea one. We define the functions $C_-(f,\cdot)$ and $C_+(f,\cdot)$ as 
\begin{equation}
    C_-(f,s)=\norm{[\DD]\{\mathrm{min}(f,s)\}}_{\M},
\end{equation}
\begin{equation}
    C_+(f,s)=\norm{[\DD]\{\mathrm{max}(f-s,0)\}}_{\M}.
\end{equation}
In Lemma \ref{lemma1}, we adapt Theorem \ref{th:coarea} for the operator $\DD$ and for piecewise constant functions.
\begin{Lemma}
\label{lemma1}
Let $f\in\M_{\DD,0}^+(\K)$ be such that $[\DD]\{f\}$ is a discrete measure with finitely many atoms. Then, the functions $C_-(f,\cdot)$ and $C_+(f,\cdot)$ are continuous, respectively increasing and decreasing, and satisfy the equalities
\begin{equation}
    C_-(f,s)=\int_{0}^s\norm{[\DD]\{\mathbbm{1}_{\{x:f(x)\geq\alpha\}}\}}_{\M}\mathrm{d}\alpha,
\end{equation}
\begin{equation}
     C_+(f,s)=\int_{s}^{\infty}\norm{[\DD]\{\mathbbm{1}_{\{x:f(x)\geq\alpha\}}\}}_{\M}\mathrm{d}\alpha.
\end{equation}
In addition, they verify the cocorner formula 
\begin{align}
\label{eq:2.C.53}
    \norm{[\DD]\{f\}}_{\M}=&\,C_-(f,s)+C_+(f,s)\nonumber\\
    =&\int_{0}^{\infty}\norm{[\DD]\{\mathbbm{1}_{\{x:f(x)\geq\alpha\}}\}}_{\M}\mathrm{d}\alpha.
\end{align}    
\end{Lemma}

\begin{proof}[\textbf{Proof of Lemma \ref{lemma1}}]
We set $[\DD]\{f\}=m=\sum_{k=1}^Ka_k\delta_{\bb{x}_k}$ and define
\begin{equation}
    \tilde{C}_-(f,s)=\int_{0}^s\norm{[\DD]\{\mathbbm{1}_{\{x:f(x)\geq\alpha\}}\}}_{\M}\mathrm{d}\alpha.
\end{equation}
Observe that $f$ takes finitely many values $(v_0,v_1,\ldots,v_M)$ with $v_m \leq v_{m+1}$, $v_0=0$ and $v_{M+1}=\infty$. Then, for $\alpha \in ]v_{m}, v_{m+1}[$ one has that 
\begin{equation}
    \{x : f(x) \geq \alpha\} = \{x : f(x) > v_{m}\}.
\end{equation}
For $s \in ]v_{m'}, v_{m'+1}[$ with $m' \in [0 \cdots M-1]$, we calculate that 
\begin{align}
\label{eq:2.1.11}
    \tilde{C}_-(f,s) = & \sum_{m=1}^{m'} (v_m - v_{m-1}) \norm{[\DD]\{\mathbbm{1}_{\{x : f(x) > v_{m-1}\}}\}}_{\M} \nonumber \\
    & + (s - v_{m'}) \norm{[\DD]\{\mathbbm{1}_{\{x : f(x) > v_{m'}\}}\}}_{\M}.
\end{align}
It follows directly from the representation \eqref{eq:2.1.11} that $\tilde{C}_-(f,\cdot)$ is increasing and continuous. Next, we observe that $[\DD]\{\mathbbm{1}_{\{x : f(x) > \alpha\}}\}$ may only be non-zero on $\bb{x}_k$. In addition, for $\alpha < \alpha'$, $\bb{x}_k$ is in the support of both $[\DD]\{\mathbbm{1}_{\{x : f(x) > \alpha'\}}\}$ and $[\DD]\{\mathbbm{1}_{\{x : f(x) > \alpha\}}\}$ if and only if it is a corner of $\mathbbm{1}_{\{x : f(x) > \alpha'\}}$ and $\mathbbm{1}_{\{x : f(x) > \alpha\}}$, in which case the Dirac mass centred at $\bb{x}_k$ must have an amplitude of the same sign because $\{x : f(x) > \alpha'\} \subset \{x : f(x) > \alpha\}$. It follows from this remark and \eqref{eq:2.1.11} that
\begin{multline}
    \tilde{C}_-(f,s) = \Bigg\| [\DD]\left\{\sum_{m=1}^{m'} (v_m - v_{m-1}) \mathbbm{1}_{\{x : f(x) > v_{m-1}\}}\right\} \\
    + [\DD]\{(s - v_{m'}) \mathbbm{1}_{\{x : f(x) > v_{m'}\}}\} \Bigg\|_{\M}
\end{multline}
and, consequently, that $\tilde{C}_-(f,s) = C_-(f,s).$ A similar argumentation would show that $C_+(f,\cdot)$ is continuous, decreasing, and that 
\begin{align}
     C_+(f,s)=\int_{s}^{\infty}\norm{[\DD]\{\mathbbm{1}_{\{x:f(x)\geq\alpha\}}\}}_{\M}\mathrm{d}\alpha.
\end{align}
Finally, 
\begin{align}
    C_-(f,s)+C_+(f,s)&=\int_{0}^{\infty}\norm{[\DD]\{\mathbbm{1}_{\{x:f(x)\geq\alpha\}}\}}_{\M}\mathrm{d}\alpha\nonumber\\
    &=\underset{s\to\infty}{\text{lim}}\norm{[\DD]\{\text{min}(f,s)\}}_{\M}\nonumber\\
    &=\norm{[\DD]\{f\}}_{\M}.
\end{align}
\end{proof}

In Lemma \ref{lemma2} we extend, with a limit argument, Lemma \ref{lemma1} for arbitrary functions in $\M_{\DD,0}^+(\K).$
\begin{Lemma}
\label{lemma2}
If $f\in\M_{\DD,0}^+(\K)$, then the functions 
$C_-(f,\cdot)$ and $C_+(f,\cdot)$ are continuous and respectively increasing and decreasing. In addition, they verify the equality 
\begin{equation}
\label{eq:2.A.22}
    \forall s>0:\quad \norm{[\DD]\{f\}}=C_-(f,s)+C_+(f,s).
\end{equation}
\end{Lemma}

\begin{proof}[\textbf{Proof of Lemma \ref{lemma2}}]
\hfill\\
\textbf{Step 1.} Consider a sequence of functions $(f_n)_{n=1}^{\infty}$ converging to $f$ as in Lemma \ref{lemma:approximation}. Without loss of generality, we assume that $f_n \in \M_{\DD,0}^+(\K)$. Our Claim 1 is that, up to a subsequence,
\begin{equation}
\label{eq:2.A.26}
    \underset{n \to \infty}{\text{lim}} C_-(f_n, s) = C_-(f, s),
\end{equation}
and
\begin{equation}
\label{eq:2.A.27}
    \underset{n \to \infty}{\text{lim}} C_+(f_n, s) = C_+(f, s).
\end{equation}
To prove this, we first recall that 
\begin{equation}
    \text{min}(f_n, s) = \frac{f_n + s - |f_n - s|}{2}.
\end{equation}
Item 2 of Lemma \ref{lemma:approximation} implies that $f_n + s$ and $f_n - s$ converge in $\mathcal{L}_1(\K)$ to $f + s$ and $f - s$, which, together with the triangle inequality, yields the convergence of $|f_n - s|$ to $|f - s|$. Therefore, $\text{min}(f_n, s)$ converges in $\mathcal{L}_1(\K)$ to $\text{min}(f, s)$. In particular,
\begin{multline}
\label{eq:2.A.19}
    \forall \psi \in \mathcal{L}_{\infty}([0,1]) \otimes \mathcal{L}_{\infty}([0,1]) \subset \mathcal{L}_{\infty}([0,1]^2): \\
    \underset{n \to \infty}{\text{lim}} \langle \text{min}(f_n, s), \psi \rangle = \langle \text{min}(f, s), \psi \rangle.
\end{multline}
Remark that 
\begin{equation}
\label{eq:remark.74}
    \Dd\{\C_c^{\infty}(\R)\}\vert_{[0,1]} \otimes \Dd\{\C_c^{\infty}(\R)\}\vert_{[0,1]} \subset \mathcal{L}_{\infty}([0,1]) \otimes \mathcal{L}_{\infty}([0,1])
\end{equation}
and that $\Dd\{\C_c^{\infty}(\R)\} \otimes \Dd\{\C_c^{\infty}(\R)\}$ is dense in $[\DD]\{\C_0(\R^2)\}$. Furthermore, it follows from Item 2 of Lemma \ref{lemma:approximation} that
\begin{align}
    \norm{[\DD]\{\text{min}(f_n, s)\}}_{\M} &\leq \norm{[\DD]\{f_n\}}_{\M}\nonumber\\
    &\leq \norm{[\DD]\{f\}}_{\M} \label{eq:truc}.
\end{align}
Consequently, the sequence $(\text{min}(f_n, s))_{n=1}^{\infty}$ is uniformly bounded and, from \eqref{eq:remark.74}, $\w$-convergent in $\M_{\DD,0}^{+}(\K)$ to $\text{min}(f, s)$. A similar argument shows that the sequence $(\text{max}(f_n - s, 0))_{n=1}^{\infty}$ is uniformly bounded and $\w$-convergent in $\M_{\DD,0}^{+}(\K)$ to $\text{max}(f - s, 0)$. It follows that 
\begin{equation}
\label{eq:2.C.101}
    \norm{[\DD]\{\text{min}(f, s)\}}_{\M} \leq \underset{n \to \infty}{\text{liminf }} C_-(f_n, s),
\end{equation}
and that
\begin{equation}
\label{eq:2.C.102}
    \norm{[\DD]\{\text{max}(f - s, 0)\}}_{\M} \leq \underset{n \to \infty}{\text{liminf }} C_+(f_n, s).
\end{equation}
In addition, 
\begin{align}
    C_{-}(f_n,s)+C_{+}(f_n,s)&=\norm{[\DD]\{f_n\}}_{\M}\\&\leq\norm{[\DD]\{f\}}_{\M}.\label{eq:2.C.103}
\end{align}
Finally, if either \eqref{eq:2.C.101} or \eqref{eq:2.C.102} has a strict inequality, we conclude that 
\begin{align}
    &\norm{[\DD]\{f\}}_{\M}\nonumber\\
    \leq&\norm{[\DD]\{\text{min}(f,s)\}}_{\M}+\norm{[\DD]\{\text{max}(f-s,0)\}}_{\M}\label{eq:2.C.105}\\
    <&\underset{n\to\infty}{\text{ liminf }} C_-(f_n,s)+\underset{n\to\infty}{\text{liminf }} C_+(f_n,s)\label{eq:2.C.106}\\
    =&\norm{[\DD]\{f\}}_{\M}\label{eq:2.C.last},
\end{align}
which is a contradiction. In \eqref{eq:2.C.105} we used the fact that 
\begin{equation}
    f=\text{min}(f,s)+\text{max}(f-s,0),
\end{equation}
and the triangular inequality. In \eqref{eq:2.C.last} we used \eqref{eq:2.C.103}. Consequently, \eqref{eq:2.C.101} and \eqref{eq:2.C.102} hold with an equality and, up to a subsequence, \eqref{eq:2.A.26} and \eqref{eq:2.A.27} hold.
\hfill\\
\textbf{Step 2.} The function  $C_-(f,\cdot)$ is increasing. Indeed, Lemma \ref{lemma1} yields that, $\forall n\in\mathbb{N}$ and $\forall s_1\leq s_2$,
\begin{align}
    &C_-(f_n,s_1)\leq C_-(f_n,s_2)
    \Rightarrow C_-(f,s_1)\leq C_-(f,s_2),
\end{align}
where we used Claim 1 to pass to the limit. The claim that $C_+(f,\cdot)$ is decreasing and \eqref{eq:2.A.22} are proved likewise.
\hfill\\
\textbf{Step 3.} We now show that $C_-(f, \cdot)$ and $C_+(f, \cdot)$ are continuous. Let $(s_n)_{n=1}^{\infty}$ be a sequence convergent to $s$. An argument similar to the one in Step 1 yields that (in parallel with \eqref{eq:2.C.101} and \eqref{eq:2.C.102})
\begin{align}
    C_-(f, s) &= \underset{n \to \infty}{\mathrm{liminf}} \quad C_-(f, s_n), \nonumber \\
    C_+(f, s) &= \underset{n \to \infty}{\mathrm{liminf}} \quad C_+(f, s_n),
\end{align}
from which it follows that 
\begin{align}
    \norm{[\DD]\{f\}}_{\M} &= \underset{n \to \infty}{\mathrm{liminf}} \quad C_-(f, s_n) + C_+(f, s_n) \nonumber \\
    &\leq \underset{n \to \infty}{\mathrm{limsup}} \quad C_-(f, s_n) + C_+(f, s_n) \label{eq:2.C.94} \\
    &= \underset{n \to \infty}{\mathrm{limsup}} \norm{[\DD]\{f\}}_{\M} \label{eq:2.C.95} \\
    &= \norm{[\DD]\{f\}}_{\M}.
\end{align}
Equation \eqref{eq:2.C.95} follows from \eqref{eq:2.A.22}, and the inequality in \eqref{eq:2.C.94} cannot be strict. It follows that
\begin{align}
    C_-(f, s) &= \underset{n \to \infty}{\mathrm{liminf}} \quad C_-(f, s_n) = \underset{n \to \infty}{\mathrm{limsup}} \quad C_-(f, s_n), \nonumber \\
    C_+(f, s) &= \underset{n \to \infty}{\mathrm{liminf}} \quad C_+(f, s_n) = \underset{n \to \infty}{\mathrm{limsup}} \quad C_+(f, s_n).
\end{align}

\end{proof}

\subsection{Extreme Point Characterization}
We are now in a position to state our main result, Theorem \ref{th:extreme}.
\begin{Theorem}
\label{th:extreme}
The extreme points of the unit ball in $\M_{\DD,0}^+(\K)$, equipped with $\norm{\cdot}_{\theta}$,  
are exactly of the form $a\mathbbm{1}_A$ with $a = \frac{1}{\norm{\mathbbm{1}_A}_{\theta}}$ and,
\begin{itemize}
    \item if $\theta = 1,$ then $A\in\mathcal{E}_{\DD}$;
    \item if $\theta \in ]0,1[,$ then $A \in \mathcal{E}_{\nabla}.$
\end{itemize}
\end{Theorem}

\begin{proof}[\textbf{Proof of Theorem \ref{th:extreme}}]
\hfill\\
\textbf{Step 1.} We show that an extreme point must be of the form $a\mathbbm{1}_A$.  
Let $f$ be an extreme point and consider the functions
\begin{equation}
    F_-(s) = \theta C_-(f,s) + (1-\theta) P_-(f,s).
\end{equation}
\begin{equation}
    F_+(s) = \theta C_+(f,s) + (1-\theta) P_+(f,s).
\end{equation}
We know that $F_-(\cdot)$ and $F_+(\cdot)$ are continuous and respectively increasing and decreasing.  
In addition, they satisfy the equality
\begin{equation}
    1 = \norm{f}_{\theta} = F_-(s) + F_+(s).
\end{equation}
By continuity, there exists $s_1$ such that $F_-(s_1) = F_+(s_1) = 0.5$, with the decomposition
\begin{equation}
    f = \frac{1}{2} \, 2\text{min}(f,s_1) + \frac{1}{2} \, 2\text{max}(f-s_1,0).
\end{equation}
We find that 
\begin{equation}
  \norm{2\text{min}(f,s_1)}_{\theta} = 2F_-(f,s_1) = 1
\end{equation}
and that
\begin{equation}
  \norm{2\text{max}(f - s_1, 0)}_{\theta} = 2F_+(f,s_1) = 1.
\end{equation}
Since $f$ is an extreme point, we must have that 
\begin{equation}
    2\text{max}(f - s_1, 0) = 2\text{min}(f, s_1).
\end{equation}
If $f \geq s_1$, then $f = 2 s_1$. If $f \leq s_1$, then $f = 0.$  
These two equations imply that $f$ has to be of the form $a \mathbbm{1}_A$ with $a = \frac{1}{\norm{\mathbbm{1}_A}_{\theta}}.$  
Finally, if $\theta \in ]0,1[,$ we assume by contradiction that $E$ is not $\nabla$-indecomposable. Then, it is not $[\DD]$-indecomposable and there exist two sets $A_1, A_2$, such that $E = A_1 \cup A_2$, 
\begin{multline}
    \norm{[\DD]\{\mathbbm{1}_{A}\}}_{\M} \\
    = \norm{[\DD]\{\mathbbm{1}_{A_1}\}}_{\M} + \norm{[\DD]\{\mathbbm{1}_{A_2}\}}_{\M},
\end{multline}
and
\begin{equation}
    \norm{\nabla\{\mathbbm{1}_{A}\}}_{\M^2} = \norm{\nabla\{\mathbbm{1}_{A_1}\}}_{\M^2} + \norm{\nabla\{\mathbbm{1}_{A_2}\}}_{\M^2}.
\end{equation}
We set $f_1 = \frac{\mathbbm{1}_{A_1}}{\norm{\mathbbm{1}_{A_1}}_{\theta}}$, $f_2 = \frac{\mathbbm{1}_{A_2}}{\norm{\mathbbm{1}_{A_2}}_{\theta}}$ and find that 
\begin{equation}
    a\mathbbm{1}_A = a \norm{\mathbbm{1}_{A_1}}_{\theta} \mathbbm{1}_{A_1} + a \norm{\mathbbm{1}_{A_2}}_{\theta} \mathbbm{1}_{A_2},
\end{equation}
which is in contradiction with the assumption that $a \mathbbm{1}_{A}$ is an extreme point. The case $\theta = 1$ is similar and omitted.
\hfill\\
\textbf{Step 2.} We show that $a\mathbbm{1}_{A}$ is an extreme point. Assume by contradiction that it is not. Then, there exists $\lambda \in\,]0,1[$ and two functions $f_1 \neq f_2$ in the unit ball such that 
\begin{equation}
    a\mathbbm{1}_{A} = \lambda f_1 + (1-\lambda) f_2. 
\end{equation}
Let $E \subset [0,1]^2$ be a measurable set. We define 
\begin{equation}
    \norm{\cdot}_{\theta,E} = \theta \norm{[\DD]\{\cdot\}}_{\M(E)} + (1-\theta) \norm{\nabla\{\cdot\}}_{\M(E)^2}.
\end{equation}
We must have
\begin{equation}
\label{eq:2.D.110}
\norm{a\mathbbm{1}_A}_{\theta,E} = \lambda \norm{f_1}_{\theta,E} + (1-\lambda) \norm{f_2}_{\theta,E},
\end{equation}
otherwise we would reach the contradiction 
\begin{align}
    1 &= \norm{a\mathbbm{1}_A}_{\theta,E} + \norm{a\mathbbm{1}_E}_{\theta,E^c} \nonumber\\
      &< \lambda \norm{f_1}_{\theta,E} + (1-\lambda) \norm{f_2}_{\theta,E} \nonumber\\
      &\quad + \lambda \norm{f_1}_{\theta,E^c} + (1-\lambda) \norm{f_2}_{\theta,E^c} = 1,
\end{align}
which is impossible.

\indent If $\theta \in\, ]0,1[$, then \eqref{eq:2.D.110} implies that $\mathrm{supp}\{\nabla\{f_i\}\} \subset \partial A.$ In particular, $\nabla\{f_i\}$ must vanish in the interior $\overset{\circ}{A}$. It follows from the Constancy Theorem \cite{dolzmann1995microstructures}, \cite[Lemma 4.6]{bredies2020sparsity} that $f_i$ is constant on $A$. Finally, since 
$a\mathbbm{1}_A = \lambda f_1 + (1-\lambda) f_2$
with $f_1 \geq 0$ and $f_2 \geq 0$, it follows that $f_1$ and $f_2$ must vanish on $\mathbb{R}^2 \setminus A$ and satisfy $f_1 = f_2 = a \mathbbm{1}_A$. This contradicts the assumption that $f_1 \neq f_2$.

If $\theta=1$, then $\mathrm{supp}\{[\DD]\{f_1\}\}$ and $\mathrm{supp}\{[\DD]\{f_2\}\}$ must be subsets of the finite set of corners of $A$, denoted by $A_c$. We use $E_c$ with subscript $c$ to denote the set of corners $\mathrm{supp}([\DD]\{1_E\})$ of a set $E$. The functions $f_1$ and $f_2$ must be piecewise constant on $A$ and vanish outside. Assume, by contradiction, that $f_1$ and $f_2$ are not constant on $A$. Then,
\begin{equation}
f_i=\sum_{k=1}^Ka_{i,k}\mathbbm{1}_{E_k},\quad\text{with}a_{i,k}\neq a_{i,k'},
\end{equation}
and, by construction, $a_{2,k}=a-a_{1,k}$. We claim that there exists a $k\in[1\cdots K]$, wlog $k=1$, such that
\begin{equation}
\begin{cases}
A_c\subset E_{1,c}\\
\partial A\subset\partial E.
\end{cases}
\end{equation}
Assume by contradiction that there exists no such $k$. Then, there exists a corner $\tilde x$ and two indices $(k,k')$, wlog $(k,k')=(1,2)$, such that 
\begin{equation}
\begin{cases}
\tilde{x}\in\partial A\\
\tilde x\notin A_c\\
\tilde x\in E_{1,c}\\
\tilde x\in E_{2,c}.
\end{cases}
\end{equation} 
This implies that $\tilde x\notin\mathrm{supp}([\DD]\{f_i\})$ and, therefore, that one must have  $a_{i,1}=a_{i,2}$ for the associated Dirac mass to cancel each other.
We apply the same argumentation to find an index k, wlog $k=2$, such that
\begin{equation}
\begin{cases}
\partial E_1\setminus \partial A\subset \partial E_{2}\\
E_{1,c}\setminus A_c\subset E_{2,c}.
\end{cases}
\end{equation}
In particular, $E_1$ and $E_2$ share a corner $\tilde x\notin\mathrm{supp}([\DD]\{f_i\})$, which therefore must be cancelled. This happens if and only if $a_{i,1}=a_{i,2}$,
which is a contradiction. Thus, $f_1=a_{1,1}E_{1}$, $f_2=a_{2,1}E_1,$ and $a_{1,1}=a_{2,1}$.  Finally $f_{1}=f_{2}$, which is a contradiction with the initial assumption. 
 
\end{proof}
In Corollary \ref{coro:opt}, we leverage the characterisation of extreme points from Theorem \ref{th:extreme} to analyse the solution set of a data-fidelity term regularised by $\norm{\cdot}_{\theta}.$

\begin{Corollary}
\label{coro:opt}
    Let the loss functional $\mathcal{J}$ be defined as
\begin{equation}
    \mathcal{J}(\cdot) = E(\bb{y}, \langle \cdot, \bm{\nu} \rangle) + \lambda \norm{\cdot}_{\theta}
\end{equation}
with
\begin{itemize}
    \item $\bb{y} \in \R^M$ is a fixed-dimensional vector representing the available measurements;
    \item the data-fidelity functional $E : \R^M \times \R^M \to \R^+ \cup \{\infty\}$ is proper, coercive, continuous, and strictly convex in its second argument;
    \item the measurement operator $\langle \cdot, \bm{\nu} \rangle : \M_{\DD,0}^+(\K) \to \R^M$ is linear and $\mathrm{weak}^\star$ continuous.
\end{itemize}
Then, the solution set
\begin{equation}
\label{eq:2.D.133}
    \mathcal{V} = \underset{f \in \M_{\DD,0}^+(\K)}{\operatorname{argmin}} \, \mathcal{J}(f)
\end{equation}
is nonempty, convex, and $\w$-compact. Moreover, its extreme points are of the form
\begin{equation}
 \label{eq:2.D.134}
    \sum_{k=1}^K a_k \mathbbm{1}_{E_k}, \quad \text{with} \quad K \leq M,
\end{equation}
where $E_k \in \mathcal{E}_{\nabla}$ if $\theta \in ]0,1[$, and $E_k \in \mathcal{E}_{\DD}$ if $\theta = 1$.
\end{Corollary}

\begin{proof}[\textbf{Proof of Corollary \ref{coro:opt}}]
Define $\mathcal{J}^{\star}=\underset{f\in\M_{\DD,0}^+(\K)}{\text{inf}}\mathcal{J}(f)$ and consider a sequence of solutions $(f_n)_{n=1}^{\infty}$ such that 
\begin{equation}
    -\infty<\underset{n\to\infty}{\text{lim}}\mathcal{J}(f_n)=\mathcal{J}^{\star}<\infty,
\end{equation}
which exists because $\mathcal{J}$ is coercive and proper. It follows that $(f_n)_{n=1}^{\infty}$ is bounded in $\norm{\cdot}_{\theta}$ and, by norm equivalence, bounded in $\norm{[\DD]\{\cdot\}}_{\M}$. According to the Banach–Alaoglu theorem, up to a subsequence, $(f_n)_{n=1}^{\infty}$ is $\w$-convergent to some limit $f$. It follows from the $\w$-lower semicontinuity of $\mathcal{J}$ that 
\begin{equation}
   \underset{n\to\infty}{\text{lim}}\mathcal{J}(f_n)=\mathcal{J}(f)=\mathcal{J}^{\star}.
\end{equation}
Consequently, $\mathcal{V}$ is nonempty, and the characterisation of its extreme points follows from \cite[Corollary 3.8]{boyer2019representer} and Theorem \ref{th:extreme}.
\end{proof}

The norm $\norm{[\DD]\{\cdot\}}_{\M}$ penalises the number of corners of the sets $E_k$, whereas $\norm{\nabla\{\cdot\}}_{\M^2}$ penalises their perimeter. Each has a distinct regularisation effect on the solutions of the optimisation problem. Informally, the former promotes solutions with simpler sets $E_k$, while the latter favours solutions with smaller sets $E_k$. This effect is proportional to $\theta$, which, in practice, must be tuned.

We conclude this section with Proposition \ref{prop:admissible}, which identifies an important class of $\w$-continuous measurement operators $\langle\cdot,\bm{\nu}\rangle$.

\begin{Proposition}
\label{prop:admissible}
    If $\bm{\nu}=(\nu_m)_{m=1}^M$ is such that $\nu_m\in\mathcal{L}_{\infty}(\mathrm{K})$, then the operator 
    \begin{equation}
        \langle\cdot,\bm{\nu}\rangle:\M_{\DD,0}(\K)\to\R^M,\quad f\mapsto(\langle f,\nu_m\rangle)_{m=1}^{M}
    \end{equation}
    is $\w$-continuous.
\end{Proposition}

\begin{proof}[\textbf{Proof of Proposition \ref{prop:admissible}}]
From \cite{guillemet2025tensor} we have 
\begin{equation}
\label{eq:2.D.157}
    \M_{\DD,0}(\K)\subset\M_{\DD}(\K)\subset\M_{\DD}(\R^2).
\end{equation}
The space $\M_{\DD}(\R^2)$, which is essentially the space of distributions $f\in\D'(\R^2)$ such that $[\DD]\{f\}\in\M(\R^2)$, is the dual space of $\C_{\DD}(\R^2)$. In particular, it is shown in \cite{guillemet2025tensor} that $\mathcal{L}_{\infty}(\K)\subset\C_{\DD}(\R^2)$. This implies that if $\bm{\nu}=(\nu_m)_{m=1}^M$ is such that $\nu_m\in\mathcal{L}_{\infty}(\mathrm{K})$, then the operator 
    \begin{equation}
        \langle\cdot,\bm{\nu}\rangle:\M_{\DD}(\R^2)\to\R^M,\quad f\mapsto(\langle f,\nu_m\rangle)_{m=1}^{M}
    \end{equation}
    is $\w$-continuous. This, together with \eqref{eq:2.D.157}, completes the proof.
\end{proof}

\section{Resolution of the Optimization Problem}
The resolution of the problem in \eqref{eq:2.D.133} is not directly feasible, as the underlying search space $\M_{\DD,0}^{+}(\K)$ is infinite-dimensional. Moreover, the fact that the edges and corners of $E_k$ in \eqref{eq:2.D.134} may lie anywhere in $\K$ makes their estimation challenging.

\subsection{Pixel-Based Exact Discretization}
We discretize the search space into a finite-dimensional cone, of functions of the form \eqref{eq:2.D.134}, with corners that lie on the dyadic grid
\begin{equation}
    \mathcal{X}_{n}=2^{-n}[0\cdots 2^n-1]\times2^{-n}[0\cdots 2^n-1],
\end{equation}
with $(2^{n}+1)^2$ knots, and associated pixels
\begin{equation}
    E^{n}_{n_1,n_2}=\left[\frac{(n_1-1)}{2^n},\frac{n_1}{2^n}\right]\times\left[\frac{(n_2-1)}{2^n},\frac{n_2}{2^n}\right].
\end{equation}
The discretized search space $\M_{\DD,0}^+(\mathcal{X}_n)$ is
\begin{equation}
\label{eq:3.A.140}
    \left\{f\in\M_{\DD,0}^+(\K):\quad[\DD]\{f\}\in\M(\mathcal{X}_n)\right\},
\end{equation}
and the associated discretized optimization problem is as in 
\begin{equation}
\label{eq:3.A.125}
    \mathcal{V}_{n}=\underset{f\in\M_{\DD,0}^+(\mathcal{X}_{n})}{\text{argmin}}\left(E(\bb{y},\langle f,\bm{\nu}\rangle)+\lambda\norm{f}_{\theta}\right).
\end{equation}
Our objective is now to reformulate \eqref{eq:3.A.125} as an optimisation problem over $\R^{2^n\times2^n}_+$, where $\R_+=[0,\infty[$. We note that any function $f\in\M_{\DD,0}^+(\mathcal{X}_{n})$ can be uniquely expressed as 
\begin{multline} 
    f=\sum_{n_1,n_2=1,1}^{2^n,2^n}a_{n_1,n_2}\mathbbm{1}_{E_{n_1,n_2}^{n}},\\
    \bb{a}=(a_{n_1,n_2})_{n_1,n_2=1,1}^{2^n,2^n}\in\R^{2^n\times2^n}_+,
\end{multline}
which corresponds to the tensor product of B-splines of order~1. Furthermore, we introduce the synthesis operator $\mathrm{T}$
\begin{multline}
    \mathrm{T}:\M_{\DD,0}^+(\mathcal{X}_{n})\to\R_+^{2^n\times2^n},\\
    f\mapsto\left(2^{2n}\langle f,\mathbbm{1}_{E^{n}_{n_1,n_2}}\rangle\right)_{n_1,n_2=1,1}^{2^n,2^n},
\end{multline}
and its adjoint, the analysis operator $\mathrm{T}^{\star}$
\begin{multline}
    \mathrm{T}^{\star}:\R_+^{2^n\times2^n}\to\M_{\DD,0}^+(\mathcal{X}_{n}),\\
    \bb{a}\mapsto\sum_{n_1,n_2=1,1}^{2^n,2^n}a_{n_1,n_2}\mathbbm{1}_{E_{n_1,n_2}^{n}}.
\end{multline}The operators $\mathrm{T}$ and $\mathrm{T}^{\star}$ are implicitly dependent on the grid $\mathcal{X}_n$. In the sequel, we shall occasionally regard $\bb{a}$ as a sequence indexed over $\mathbb{Z}^2$, in which case its entries will be denoted by $\bb{a}[\cdot,\cdot]$ and taken to be equal to $0$ outside $[0\cdots2^n]\times[0\cdots2^n]$. We consider the new optimisation problem  
\begin{equation}   
\label{eq:3.A.144}
\tilde{\mathcal{V}}_{n}=\underset{\bb{a}\in\R^{2^n\times2^n}_+}{\mathrm{argmin}}\left(E(\bb{y},\langle \mathrm{T}^{\star}\{\bb{a}\},\bm{\nu}\rangle)+\lambda\norm{\bb{a}}_{\theta}\right),
\end{equation}
where 
\begin{multline}
    \label{eq:3.A.135}
    \norm{\bb{a}}_{\theta}=(1-\theta)2^{-n}\left(\norm{\bm{h}_{1,0}\ast\bb{a}}_{\ell_1}+\norm{\bm{h}_{0,1}\ast\bb{a}}_{\ell_1}\right)\\
    +\theta\norm{\bm{h}_{1,1}\ast\bb{a}}_{\ell_1},
\end{multline}
with 
\begin{equation}
    \bm{h}_{1,1}=\begin{bmatrix}
        1&-1\\
        -1&1
    \end{bmatrix},\quad
    \bm{h}_{1,0}=\begin{bmatrix}
        1\\
        -1
    \end{bmatrix},\quad
    \bm{h}_{0,1}=\begin{bmatrix}
        1&-1\\
    \end{bmatrix}.
\end{equation}
A key feature of this formulation, as shown in Proposition~\ref{prop:AnalysisSynthesis}, is that the problem in \eqref{eq:3.A.144} is the exact discretisation of the problem in \eqref{eq:3.A.125}, \emph{i.e.}, it incurs no discretisation error.

\begin{Proposition}
\label{prop:AnalysisSynthesis}
The optimization problem in \eqref{eq:3.A.144} is equivalent to the one in \eqref{eq:3.A.125}, in the sense that,
\begin{align}
    f\in\mathcal{V}_n&\Leftrightarrow\mathrm{T}\{f\}\in\tilde{\mathcal{V}}_n.
\end{align}
\end{Proposition}

\begin{proof}[\textbf{Proof of Proposition \ref{prop:AnalysisSynthesis}}]
Let $\bb{a}=\mathrm{T}\{f\}$. Observing that $\mathrm{T}^{\star}$ is the inverse of $\mathrm{T}$, it suffices to prove that 
\begin{align}
\norm{[\DD]\{f\}}_{\M}&=\norm{\bm{h}_{1,1}\ast\bb{a}}_{\ell_1},\label{eq:3.A.133}\\
\norm{[\mathrm{I}\otimes\Dd]\{f\}}_{\M}&=2^{-n}\norm{\bm{h}_{0,1}\ast\bb{a}}_{\ell_1},\\
\norm{[\Dd\otimes\mathrm{I}]\{f\}}_{\M}&=2^{-n}\norm{\bm{h}_{1,0}\ast\bb{a}}_{\ell_1}.
\end{align}
We prove only \eqref{eq:3.A.133}, as the reasoning for the other two is analogous.
We observe that
\begin{align} 
\mathbbm{1}_{E_{n_1,n_2}^{n}}=\beta\left(2^{n}\cdot-(n_1-1)\right)\otimes\beta\left(2^{n}\cdot-(n_2-1)\right),
\end{align}
 and calculate that 
\begin{align}
    &[\DD]\{f\}\nonumber\\
    =&\sum_{n_1,n_2=1,1}^{2^n,2^n}\bb{a}[n_1,n_2]\left(\delta_{\frac{n_1}{2^{n}}}-\delta_{\frac{(n_1-1)}{2^{n}}}\right)\otimes\left(\delta_{\frac{n_2}{2^{n}}}-\delta_{\frac{(n_2-1)}{2^{n}}}\right)\nonumber\\
    =&\sum_{(n_1,n_2)\in\mathbb{Z}^2}\bb{a}[n_1,n_2]\left(\delta_{\frac{n_1}{2^{n}}}-\delta_{\frac{(n_1-1)}{2^{n}}}\right)\otimes\left(\delta_{\frac{n_2}{2^{n}}}-\delta_{\frac{(n_2-1)}{2^{n}}}\right)\nonumber\\
=&\sum_{(n_1,n_2)\in\mathbb{Z}^2}\delta_{\frac{n_1}{2^{n}},\frac{n_2}{2^{n}}}\bigg(\bb{a}[n_1,n_2]-\bb{a}[n_1+1,n_2]\nonumber\\
&-\bb{a}[n_1,n_2+1]+\bb{a}[n_1+1,n_2+1]
    \bigg).
\end{align}
Thus 
\begin{align}
    &\norm{[\DD]\{f\}}_{\M}=\sum_{(n_1,n_2)\in\mathbb{Z}^2}\big\vert(\bm{h}_{1,1}\ast\bb{a})[n_1,n_2]\big\vert.
\end{align}
\end{proof}
Proposition \ref{prop:AnalysisSynthesis} enables one to obtain a solution $f_n^{\star}\in\mathcal{V}_n$ by solving the problem in \eqref{eq:3.A.144}. Moreover, since the operator  
\begin{equation}
\label{eq:3.A.156}
    \langle\mathrm{T}^{\star}\{\cdot\},\bm{\nu}\rangle:\R^{2^n\times2^n}_+\to\R^M
\end{equation}
is linear, when $\bb{a}$ is vectorised, the operator in \eqref{eq:3.A.156} can be represented by a matrix. Consequently, \eqref{eq:3.A.144} can be solved using a primal–dual algorithm \cite{condat2013primal} or a double-staged FISTA \cite{bonettini2018inertial}.

\subsection{Convergence}
The solution $f_n^{\star}\in\mathcal{V}_n$ obtained by solving the discretized optimisation problem \eqref{eq:3.A.125} is only an approximation of some $f^{\star}\in\mathcal{V}$. Therefore, for this exact grid-based discretisation to be reliable and practical, one must be confident that, in some sense,
\begin{equation}
\underset{n\to\infty}{\lim} f_{n}^{\star} = f^{\star}.
\end{equation}
This convergence is established in Theorem \ref{th:convergence}.

\begin{Theorem}
    \label{th:convergence}
For any sequence $(f_n^{\star})_{n=1}^{\infty}$
of discretized solutions with $f_n^{\star}\in\mathcal{V}_n,$ and for any solution $f^{\star}\in\mathcal{V}$, the
 following convergences hold:
 \begin{itemize}
     \item [1.] over the loss functional, $\underset{n\to\infty}{\mathrm{lim}}\mathcal{J}(f_n^{\star})=\mathcal{J}(f^{\star})$;
     \item [2.] over the simulated measurements, $\underset{n\to\infty}{\mathrm{lim}}\langle f_n^{\star},\bm{\nu}\rangle=\langle f^{\star},\bm{\nu}\rangle$.
 \end{itemize}
\end{Theorem} 

\begin{proof}[\textbf{Proof of Theorem \ref{th:convergence}}]
\hfill\\
\textbf{Item 1.} We fix $f^{\star}\in\mathcal{V}$ and use Lemma \ref{lemma:approximation} to find a sequence $(\tilde{f}_n)_{n=1}^{\infty}$ with $\tilde{f}_n\in\M_{\DD,0}^+(\mathcal{X}_n)$ that is $\w$-convergent to $f^{\star}$ and satisfies $\norm{f^{\star}}_{\theta}=\underset{n\to\infty}{\lim}\norm{\tilde{f}_n}_{\theta}$. Then, the $\w$-convergence implies that 
\begin{multline}
\label{eq:3.B.172}
    \underset{n\to\infty}{\lim}E(\bb{y},\langle \tilde{f}_n,\bm{\nu}\rangle) = E(\bb{y},\langle f^{\star},\bm{\nu}\rangle)\\
    \Rightarrow \underset{n\to\infty}{\lim}\mathcal{J}(\tilde{f}_n) = \mathcal{J}(f^{\star}).
\end{multline}
It follows from \eqref{eq:3.B.172} that 
\begin{equation}
    \mathcal{J}(f^{\star}) \leq \mathcal{J}(f_n^{\star}) \leq \mathcal{J}(\tilde{f}_n) \Rightarrow \underset{n\to\infty}{\lim} \mathcal{J}(f_n^{\star}) = \mathcal{J}(f^{\star}).
\end{equation}
\hfill\\
\textbf{Item 2.} The convergence in Item 1 implies that the sequence $(\mathcal{J}(f_n^{\star}))_{n=1}^{\infty}$ is bounded. In turn, this boundedness, combined with the assumptions on $E$, implies that the sequence $(f_n^{\star})_{n=1}^{\infty}$ is bounded in $\norm{\cdot}_{\theta}$, and therefore in $\norm{[\DD]\{\cdot\}}_{\M}$. By the Banach-Alaoglu theorem, we extract a $\w$-convergent subsequence $(f_{n_{m}}^{\star})_{m=1}^{\infty}$ and find that  
\begin{equation}
    \underset{m\to\infty}{\lim} \langle f_{n_m}^{\star},\bm{\nu}\rangle = \langle f^{\star\star},\bm{\nu}\rangle
\end{equation}
for some $f^{\star\star} \in \mathcal{V}$. If one assumes by contradiction that the sequence $(\langle f_{n},\bm{\nu}\rangle)_{n=1}^{\infty}$ does not converge to $\langle f^{\star\star},\bm{\nu}\rangle$, then a classical argument \cite[Theorem 3]{guillemet2025convergence}, which leverages the strict convexity of $E(\bb{y},\cdot)$, leads to a contradiction.
\end{proof}

\section{Application to Denoising}

\subsection{General Description}
Our goal is to denoise an image $\bb{y}$, which is assumed to be square and composed of pixels 
\begin{equation}
    E_{n_1,n_2}^N = \left[\frac{n_1 - 1}{2^N}, \frac{n_1}{2^N}\right] \times \left[\frac{n_2 - 1}{2^N}, \frac{n_2}{2^N}\right]
\end{equation}
at resolution $N$. The continuous-domain optimization problem is formulated as
\begin{equation}
\label{eq:4.A.121}
    \mathcal{V} = \underset{f \in \M_{\DD,0}^+(\K)}{\mathrm{argmin}} \left( \norm{\bb{y} - \bm{\nu}\{f\}}_2^2 + \lambda \norm{f}_{\theta} \right),
\end{equation}
where $\bb{y} \in \R^{2^N \times 2^N}$, and $\forall \bb{n} \in [1 \cdots 2^N]^2$,
\begin{equation}
\label{eq:4.B.155}
    [\bm{\nu}\{f\}]_{\bb{n}} = \langle f, \mathbbm{1}_{E_{\bb{n}}^N} \rangle = \int_{E_{\bb{n}}^N} f(\bb{x}) \mathrm{d}\bb{x}.
\end{equation}

We know from Corollary \ref{coro:opt} that $\mathcal{V}$ is nonempty, compact, and convex, and that its extreme points are of the form $\sum_{k=1}^K a_k \mathbbm{1}_{E_k}$ with $K \leq 2^{2N}$. The discretized optimization problem is defined as
\begin{equation}
\label{eq:4.A.124}
    \mathcal{V}_{n} = \underset{f \in \M_{\DD,0}^+(\mathcal{X}_n)}{\mathrm{argmin}} \left( \norm{\bb{y} - \bm{\nu}\{f\}}_2^2 + \lambda \norm{f}_{\theta} \right),
\end{equation}
and a solution $f^{\star}_n \in \mathcal{V}_n$ can be obtained by solving \eqref{eq:3.A.144}. Finally, from Theorem \ref{th:convergence}, $\forall f^{\star} \in \mathcal{V}$ and $\forall\bb{n} \in [1 \cdots 2^N]^2$, we obtain
\begin{equation}
\label{eq:4.B.155}
    \underset{n \to \infty}{\lim} \int_{E_{\bb{n}}^N} f_n(\bb{x}) \, \mathrm{d}\bb{x} = \int_{E_{\bb{n}}^N} f^{\star}(\bb{x}) \, \mathrm{d}\bb{x}.
\end{equation}

\subsection{Resolution on Finite Grid}
We argue that, to solve the continuous-domain problem in \eqref{eq:4.A.121}, it is sufficient to find the solution $f^{\star}_N \in \mathcal{V}_N$. As a first step, we show in Theorem \ref{th:finitegrid} that the sets of solutions $(\mathcal{V}_n)_{n=N}^{\infty}$ are embedded one into another.

\begin{Theorem}
\label{th:finitegrid}
For all $N_2\geq N_1\geq N$, the inclusions hold true
\begin{equation}
 \mathcal{V}_{N_1}\subset\mathcal{V}_{N_2}\subset\mathcal{V}.   
\end{equation}
\end{Theorem}

\begin{proof}[\textbf{Proof of Theorem \ref{th:finitegrid}}]
We define the downsampling operator
\begin{align}
 \mathrm{R}:\M_{\DD,0}^+(\mathcal{X}_{n})&\to\M_{\DD,0}^+(\mathcal{X}_{n-1})\nonumber\\
    \sum_{\bb{n}=1,1}^{2^n,2^n}\bb{a}[\bb{n}]\mathbbm{1}_{E_{\bb{n}}^{n}}&\mapsto\sum_{\bb{n}=1,1}^{2^{n-1},2^{n-1}}\bb{b}[\bb{n}]\mathbbm{1}_{E_{\bb{n}}^{(n-1)}}
\end{align}
with $\bb{b}[\bb{n}]=\bb{b}[n_1,n_2]$ equal to
\begin{multline}
\label{eq:4.B.131}
  \frac{1}{4}\big(\bb{a}[2n_1-1,2n_2-1]+\bb{a}[2n_1-1,2n_2]\\+\bb{a}[2n_1,2n_2-1]+\bb{a}[2n_1,2n_2]\big).
\end{multline}
\textbf{Step 1.}
We claim that, $\forall f\in\M_{\DD,0}^+(\mathcal{X}_{N_1})$ with $N_1> N$,
\begin{equation}
\label{eq:4.B.132}
\mathcal{J}(\mathrm{R}\{f\})\leq\mathcal{J}(f). 
\end{equation}
A straightforward calculation yields that $\langle f,\mathbbm{1}_{E_{\bb{n}}^N}\rangle$ is equal to $\langle \mathrm{R}\{f\},\mathbbm{1}_{E_{\bb{n}}^N}\rangle$ and, consequently, that 
\begin{equation}
\label{eq:4.B.133}
\norm{\bb{y}-\bm{\nu}\{f\}}_2^2=\norm{\bb{y}-\bm{\nu}\{\mathrm{R}\{f\}\}}_2^2.
\end{equation}
Next, we get from \eqref{eq:3.A.133} that 
\begin{align}
    \norm{[\DD]\{f\}}_{\M}&=\sum_{\bb{n}\in\mathbb{Z}^2}\vert (\bm{h}_{1,1}\ast \bb{a})[\bb{n}]\vert\nonumber\\
    \norm{[\DD]\{\mathrm{R}\{f\}\}}_{\M}&=\sum_{\bb{n}\in\mathbb{Z}^2}\vert (\bm{h}_{1,1}\ast \bb{b})[\bb{n}]\vert\label{eq:4.B.135}.
\end{align}
The injection of \eqref{eq:4.B.131} in \eqref{eq:4.B.135} and the use of several triangle inequalities (9 of them) yield that 
\begin{align}
\vert\bm{h}_{1,1}\ast \bb{b}\vert[n_1,n_2]\leq&
  \frac{1}{4}\vert(\bm{h}_{1,1}\ast\bb{a})[2n_1-1,2n_2-1]\vert\nonumber\\
  +&\frac{1}{2}\vert(\bm{h}_{1,1}\ast\bb{a})[2n_1-1,2n_2]\vert\nonumber\\
  +&\frac{1}{4}\vert(\bm{h}_{1,1}\ast\bb{a})[2n_1-1,2n_2+1]\vert\nonumber\\
  +&\frac{1}{2}\vert(\bm{h}_{1,1}\ast\bb{a})[2n_1,2n_2-1]\vert\nonumber\\
  +&\vert(\bm{h}_{1,1}\ast\bb{a})[2n_1,2n_2]\vert\nonumber\\
  +&\frac{1}{2}\vert(\bm{h}_{1,1}\ast\bb{a})[2n_1,2n_2+1]\vert\nonumber\\
  +&\frac{1}{4}\vert(\bm{h}_{1,1}\ast\bb{a})[2n_1+1,2n_2-1]\vert\nonumber\\
  +&\frac{1}{2}\vert(\bm{h}_{1,1}\ast\bb{a})[2n_1+1,2n_2]\vert\nonumber\\
  +&\frac{1}{4}\vert(\bm{h}_{1,1}\ast\bb{a})[2n_1+1,2n_2+1]\vert.\label{eq:4.B.136}
\end{align}
Finally, the sum of \eqref{eq:4.B.136} over $(n_1,n_2)\in\mathbb{Z}^2$ yields that 
\begin{equation}
\label{eq:4.B.137}
    \norm{[\DD]\{\mathrm{R}\{f\}\}}_{\M}\leq\norm{[\DD]\{f\}}_{\M},
\end{equation}
and a similar calculation would show that 
\begin{equation}
\label{eq:4.B.138}
    \norm{\nabla\{\mathrm{R}\{f\}\}}_{\M^2}\leq\norm{\nabla\{f\}}_{\M^2}.
\end{equation}
The claim \eqref{eq:4.B.132} follows from the combination of \eqref{eq:4.B.133}, \eqref{eq:4.B.137}, and \eqref{eq:4.B.138}.
\hfill\\
\textbf{Step 2.} 
Let $f_{N_1}\in\mathcal{V}_{N_1}$ and $f_{N_2}\in\mathcal{V}_{N_2}$. One can iterate Step 1 to find that 
\begin{equation}
\label{eq:4.B.139}
    \mathcal{J}(f_{N_1})\leq\mathcal{J}(\mathrm{R}^{N_2-N_1}\{f_{N_2}\})\leq\mathcal{J}(f_{N_2}),
\end{equation}
where the LHS inequality follows from the inclusion $\mathrm{R}^{N_2-N_1}\{f_{N_2}\}\in\M_{\DD,0}^+(\mathcal{X}_{N_1})$ and the RHS inequality follows from \eqref{eq:4.B.132}. In addition, it follows from the inclusion 
\begin{equation}
    \M_{\DD,0}^+(\mathcal{X}_{N_1})\subset\M_{\DD,0}^+(\mathcal{X}_{N_2})
\end{equation}
that $\mathcal{J}(f_{N_2})\leq\mathcal{J}(f_{N_1})$, which, combined with \eqref{eq:4.B.139}, yields that $\mathcal{J}(f_{N_2})=\mathcal{J}(f_{N_1})$. This proves the inclusion $\mathcal{V}_{N_1}\subset\mathcal{V}_{N_2}.$
Finally, if $(f_n^{\star})_{n=1}^{\infty}$ is a sequence of solutions with $f_n\in\mathcal{V}_n$, then for some $f^{\star}\in\mathcal{V}$,
\begin{equation}
\label{eq:4.B.141}
    \underset{n\to\infty}{\text{lim}}\mathcal{J}(f_{n}^{\star})=\mathcal{J}(f^{\star}).
\end{equation}
One can iterate \eqref{eq:4.B.132} to find that, $\forall n\geq N_2$,
\begin{equation}
    \mathcal{J}(f^{\star})\leq\mathcal{J}(f_{N_2})\leq\mathcal{J}(\mathrm{R}^{n-N_2}\{f_n\})\leq\mathcal{J}(f_n),\label{eq:152}
\end{equation}
The combination of Equation \eqref{eq:152} and \eqref{eq:4.B.141} yields the inclusion $\mathcal{V}_{N_2}\subset\mathcal{V}.$
\end{proof}Theorem \ref{th:finitegrid} reveals that a solution of the continuous-domain problem can be found by solving the optimization problem on $\M_{\DD,0}^+(\mathcal{X}_N)$. The optimization over $\M_{\DD,0}^+(\mathcal{X}_N)$ has a unique solution.

\begin{Proposition}
\label{prop:unique}
There is a unique solution $f^{\star} \in \mathcal{V}$ whose knots lie on the grid $\mathcal{X}_{N}$, \emph{i.e.},
\begin{equation}
    \mathcal{V}_N = \{f^{\star}\}.
\end{equation}
\end{Proposition}

\begin{proof}[\textbf{Proof of Proposition \ref{prop:unique}}]
    We know from Theorem \eqref{th:convergence} that there is a unique value $\bb{y}^{\star} \in \R^{2^N \times 2^N}$ such that, 
    $\forall f^{\star} \in \mathcal{V}_N \subset \mathcal{V}$,
    \begin{equation}
        \bb{y}^{\star} = \bm{\nu}\{f^{\star}\}.
    \end{equation}
    We conclude with the observation that, on $\M_{\DD,0}^{+}(\mathcal{X}_N)$, the measurement operator $\bm{\nu}$ is a bijection. 
\end{proof}

The combination of Theorem \ref{th:finitegrid} and Proposition \ref{prop:unique} yields that \eqref{eq:3.A.144}, specialised to the denoising task, has a unique solution which also happens to be a solution (through the synthesis operator) of the continuous-domain optimization problem.

\subsection{Numerical Experiments}
We fix the data fidelity as the squared difference and study the optimization problem in
\begin{equation}
    \mathcal{V}_{\lambda,\theta} = \underset{f \in \M_{\DD,0}^+(\K)}{\text{argmin}} \left( \frac{1}{2} \norm{\bb{y} - \bm{\nu}\{f\}}_2^2 + \lambda \norm{f}_{\theta} \right),
\end{equation}
where $\bm{\nu}$ is defined in \eqref{eq:4.B.155} and $\bb{y} \in \R^{2^N \times 2^N}$ is the image to denoise. We discretize this optimization problem, as in \eqref{eq:4.A.124}, on the canonical pixel basis $\mathcal{X}_N$ and denote by $f_{\lambda,\theta}^{\star} \in \mathcal{V}_{\lambda,\theta}$ its unique solution (Theorem \ref{th:finitegrid} and Proposition \ref{prop:unique}). Then, we consider the synthesis formulation \eqref{eq:3.A.144} of the discretized optimization problem in
\begin{equation}   
\label{eq:4.C.176}
    \{\bb{a}_{\lambda,\theta}^{\star}\} = \underset{\bb{a} \in \R^{2^{N} \times 2^{N}}_+}{\text{argmin}} \left( \frac{1}{2} \norm{\bb{y} - \bb{a}}_2^2 + \lambda \norm{\bm{h}_{\theta} \ast \bb{a}}_{1} \right),
\end{equation}
whose regularization term has been reparametrized for convenience, with
\begin{equation}
    \bm{h}_{\theta} = \left[\begin{array}{rr}
        \theta \bm{h}_{1,1} \\
        \left(1 - \frac{\theta}{2}\right) \bm{h}_{1,0} \\
        \left(1 - \frac{\theta}{2}\right) \bm{h}_{0,1} \\
    \end{array}\right].
\end{equation}
The convolution in \eqref{eq:4.C.176} is applied element wise on the rows of $\bm{h}_{\theta}.$
Finally, since there is no closed form for the proximal operator of $\lambda \norm{\bm{h}_{\theta} \ast \bb{a}}_{1}$, we solve the dual problem. The latter is optimised with the FISTA algorithm \cite{beck2009fast}. In practice, our method extends without any issues to images that are not square. The denoising efficiency of our method, labelled MTV, is assessed on the BSD68 test set alongside two other benchmark methods. The results are summarised in Table \ref{tab:numbers}.

\begin{table}[h!]
    \centering
    \begin{tabular}{c|c c c c}
    \hline\hline
         & $\sigma=5/255$ & $\sigma=15/255$ & $\sigma=25/255$ &\\ \hline
         TV & 36.41 & 29.91 & 27.48\\
         MTV & 36.83 & 30.20 & 27.72\\
         CRR-NN \cite{goujon2023neural} & 36.96 & 30.55 & 28.11\\
         \hline\hline
    \end{tabular}
    \caption{Denoising performance on the BSD68 test set. The chosen metric is the PSNR, averaged over the 68 images. The parameters $\lambda$ and $\theta$ for TV and MTV have been optimised to maximise the PSNR. The method CRR-NN is included for comparison. In this method, the authors trained 24 filters of size $15 \times 15$, parametrised by neural networks, on a Gaussian-denoising task.
}
    \label{tab:numbers}
\end{table}
To investigate the choice of the hyper-parameter $\theta$, we denote by $\theta^{\star}_{\sigma}$ the value of $\theta$ that maximises the PSNR for a specific image and noise level $\sigma$, and provide in Table \ref{tab:theta} statistics on the distribution of $\theta^{\star}_{\sigma}$.
\begin{table}[h!]
    \centering
    \begin{tabular}{c|c c c c}
    \hline\hline
         & $\theta^{\star}_{5}$ & $\theta^{\star}_{15}$ & $\theta^{\star}_{25}$ & $\vert\theta^{\star}_{15}-\theta^{\star}_{25}\vert$ \\ \hline
        Expectation & 0.37 & 0.33 & 0.33 & 0.04\\
         Standard deviation & 0.07 & 0.08 & 0.12 & 0.04 \\
         \hline\hline
    \end{tabular}
    \caption{Statistics of $\theta^{\star}_{\sigma}$ taken over 68 samples, corresponding to the optimal values of $\theta$ for the 68 images of the BSD68 test set, for each noise level.}
    \label{tab:theta}
\end{table}

We observe from the expectations in Table \ref{tab:theta} that our method does not reduce to the TV one ($\theta=0$). The optimal convex combination of $\nabla$-based and $[\DD]$-based regularizations is non-trivial and appears necessary for efficient denoising of natural images. Furthermore, the column $\vert\theta^{\star}_{15}-\theta^{\star}_{25}\vert$ shows that the choice of the optimal $\theta$ is stable with respect to $\sigma$, which, in turn, implies that the optimal $\theta$ depends mostly on the geometry of the image rather than on the noise level. 

We observe that images in BSD68 for which MTV significantly outperforms TV, by more than $0.5$ dB, typically feature corner-like structures, such as in Image \ref{fig:21}.

\begin{figure}[h]
    \centering
    \includegraphics[width=0.5\linewidth]{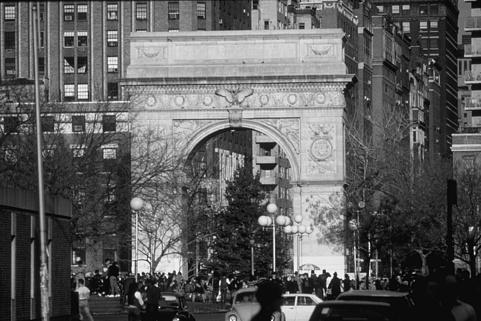}
    \caption{Image 21 in BSD68 dataset. For $\sigma=25$, the PSNR for MTV and TV are, respectively, $25.96$ and $25.29$.}
    \label{fig:21}
\end{figure}

To better illustrate this behaviour, we consider a $722\times 1920$ image of New York City, free of rights, as shown in Figure \ref{fig:ny}. For the noise level $\sigma=25/255$, the TV and MTV methods yield PSNR values of 24.70 and 25.52, respectively. We observe that bands of constant values and corners are reconstructed better with MTV.

\begin{figure*}[tb]
\centering
\label{fig:ny}
\includegraphics[width=1\linewidth]{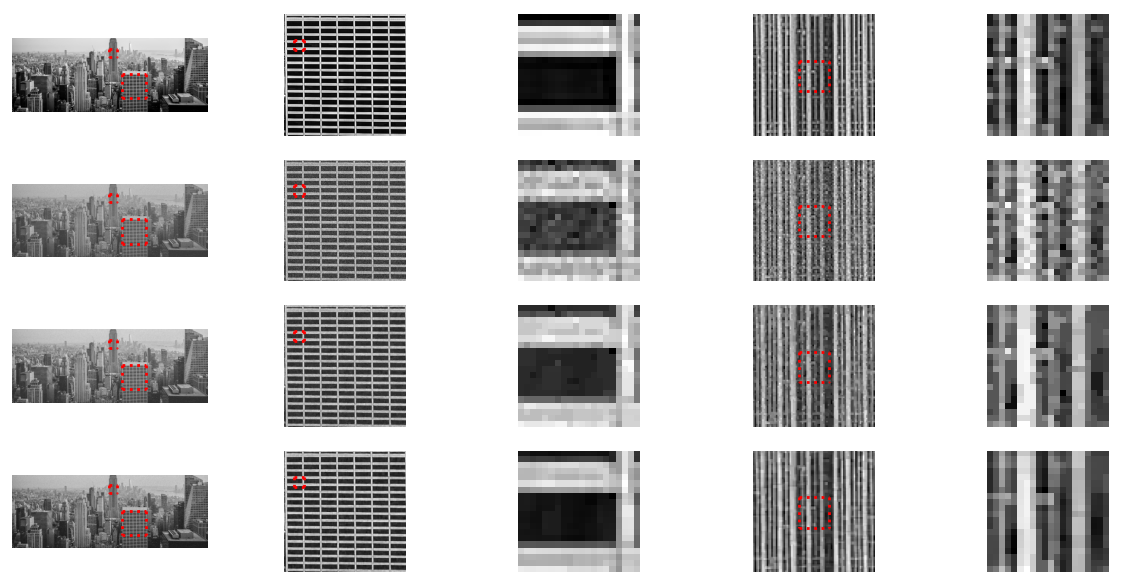}
\caption{From top to bottom, each line corresponds to the ground truth, the noisy image, the TV-denoised image, and the MTV-denoised image. Columns 2 and 4 are zooms of the ground truth. Columns 3 and 5 are zooms of Columns 2 and 4.
}
\end{figure*}

\section{Conclusion}
In this paper, we have shown that the convex combination of the TV regularization with the tensor product term $\norm{[\DD]\{\cdot\}}_{\M}$ promotes piecewise constant functions with edges parallel to the $x_1$ or the $x_2$ axis, as solutions to a continuous-domain optimization problem. In the particular case of denoising, we revealed that the optimization problem discretized on the canonical, pixel-based grid has a unique solution, which is also a solution of the continuous-domain problem. Finally, we demonstrated that the discretized problem can be solved with the same algorithm as for TV, with an additional filter, and that our new regularization significantly improves the reconstruction quality when the image has the appropriate underlying geometry.

\section*{Acknowledgment}
Vincent Guillemet was supported by the Swiss National Science Foundation (SNSF) under Grant 200020\_219356.
\ifCLASSOPTIONcaptionsoff
  \newpage
\fi

\bibliographystyle{IEEEtran}
\bibliography{ref}

% Generated by IEEEtran.bst, version: 1.14 (2015/08/26)
\begin{thebibliography}{10}
\providecommand{\url}[1]{#1}
\csname url@samestyle\endcsname
\providecommand{\newblock}{\relax}
\providecommand{\bibinfo}[2]{#2}
\providecommand{\BIBentrySTDinterwordspacing}{\spaceskip=0pt\relax}
\providecommand{\BIBentryALTinterwordstretchfactor}{4}
\providecommand{\BIBentryALTinterwordspacing}{\spaceskip=\fontdimen2\font plus
\BIBentryALTinterwordstretchfactor\fontdimen3\font minus \fontdimen4\font\relax}
\providecommand{\BIBforeignlanguage}[2]{{%
\expandafter\ifx\csname l@#1\endcsname\relax
\typeout{** WARNING: IEEEtran.bst: No hyphenation pattern has been}%
\typeout{** loaded for the language `#1'. Using the pattern for}%
\typeout{** the default language instead.}%
\else
\language=\csname l@#1\endcsname
\fi
#2}}
\providecommand{\BIBdecl}{\relax}
\BIBdecl

\bibitem{bertero2021introduction}
M.~Bertero, P.~Boccacci, and C.~De~Mol, \emph{Introduction to Inverse Problems in Imaging}.\hskip 1em plus 0.5em minus 0.4em\relax CRC press, 2021.

\bibitem{mccann2019biomedical}
M.~T. McCann, M.~Unser \emph{et~al.}, ``Biomedical image reconstruction: the foundations to deep neural networks,'' \emph{Foundations and Trends{\textregistered} in Signal Processing}, vol.~13, no.~3, pp. 283--359, 2019.

\bibitem{gupta2018continuous}
H.~Gupta, J.~Fageot, and M.~Unser, ``Continuous-domain solutions of linear inverse problems with {Tikhonov} versus generalized {TV} regularization,'' \emph{IEEE Transactions on Signal Processing}, vol.~66, no.~17, pp. 4670--4684, 2018.

\bibitem{unser2019native}
M.~Unser and J.~Fageot, ``Native {Banach} spaces for splines and variational inverse problems,'' \emph{arXiv preprint arXiv:1904.10818}, 2019.

\bibitem{rudin1992nonlinear}
L.~I. Rudin, S.~Osher, and E.~Fatemi, ``Nonlinear total variation based noise removal algorithms,'' \emph{Physica D: Nonlinear Phenomena}, vol.~60, no. 1-4, pp. 259--268, 1992.

\bibitem{ambrosio2000functions}
L.~Ambrosio, N.~Fusco, and D.~Pallara, \emph{Functions of Bounded Variation and Free Discontinuity Problems}.\hskip 1em plus 0.5em minus 0.4em\relax Oxford University Press, 2000.

\bibitem{evans2018measure}
L.~Evans, \emph{Measure Theory and Fine Properties of Functions}.\hskip 1em plus 0.5em minus 0.4em\relax Routledge, 2018.

\bibitem{rudin}
W.~Rudin, \emph{Real and Complex Analysis}, 3rd~ed.\hskip 1em plus 0.5em minus 0.4em\relax New York: McGraw-Hill, 1987.

\bibitem{donoho2006compressed}
D.~L. Donoho, ``Compressed sensing,'' \emph{IEEE Transactions on Information Theory}, vol.~52, no.~4, pp. 1289--1306, 2006.

\bibitem{candes2007sparsity}
E.~Candes and J.~Romberg, ``Sparsity and incoherence in compressive sampling,'' \emph{Inverse Problems}, vol.~23, no.~3, p. 969, 2007.

\bibitem{bruckstein2009sparse}
A.~M. Bruckstein, D.~L. Donoho, and M.~Elad, ``From sparse solutions of systems of equations to sparse modeling of signals and images,'' \emph{SIAM Review}, vol.~51, no.~1, pp. 34--81, 2009.

\bibitem{beck2009fast}
A.~Beck and M.~Teboulle, ``Fast gradient-based algorithms for constrained total variation image denoising and deblurring problems,'' \emph{IEEE Transactions on Image Processing}, vol.~18, no.~11, pp. 2419--2434, 2009.

\bibitem{zhu2008efficient}
M.~Zhu and T.~Chan, ``An efficient primal-dual hybrid gradient algorithm for total variation image restoration,'' \emph{Ucla Cam Report}, vol.~34, no.~2, 2008.

\bibitem{esser2010general}
E.~Esser, X.~Zhang, and T.~F. Chan, ``A general framework for a class of first order primal-dual algorithms for convex optimization in imaging science,'' \emph{SIAM Journal on Imaging Sciences}, vol.~3, no.~4, pp. 1015--1046, 2010.

\bibitem{chambolle2011first}
A.~Chambolle and T.~Pock, ``A first-order primal-dual algorithm for convex problems with applications to imaging,'' \emph{Journal of Mathematical Imaging and Vision}, vol.~40, pp. 120--145, 2011.

\bibitem{condat2013primal}
L.~Condat, ``A primal--dual splitting method for convex optimization involving lipschitzian, proximable and linear composite terms,'' \emph{Journal of Optimization Theory and Applications}, vol. 158, no.~2, pp. 460--479, 2013.

\bibitem{caccioppoli1952misura}
R.~Caccioppoli, ``Misura e integrazione sugli insiemi dimensionalmente orientati,'' \emph{Atti Accad. Naz. Lincei. Rend. Cl. Sci. Fis. Mat. Nat.(8)}, vol.~12, pp. 3--11, 1952.

\bibitem{de1954teoria}
E.~De~Giorgi, ``Su una teoria generale della misura (r- 1)-dimensionale in uno spazio ad r dimensioni,'' \emph{Annali di Matematica Pura ed Applicata}, vol.~36, pp. 191--213, 1954.

\bibitem{flinth2019exact}
A.~Flinth and P.~Weiss, ``Exact solutions of infinite dimensional total-variation regularized problems,'' \emph{Information and Inference: A Journal of the IMA}, vol.~8, no.~3, pp. 407--443, 2019.

\bibitem{unser2017splines}
M.~Unser, J.~Fageot, and J.~P. Ward, ``Splines are universal solutions of linear inverse problems with generalized {TV} regularization,'' \emph{SIAM Review}, vol.~59, no.~4, pp. 769--793, 2017.

\bibitem{unser2021unifying}
M.~Unser, ``A unifying representer theorem for inverse problems and machine learning,'' \emph{Foundations of Computational Mathematics}, vol.~21, no.~4, pp. 941--960, 2021.

\bibitem{boyer2019representer}
C.~Boyer, A.~Chambolle, Y.~D. Castro, V.~Duval, F.~De~Gournay, and P.~Weiss, ``On representer theorems and convex regularization,'' \emph{SIAM Journal on Optimization}, vol.~29, no.~2, pp. 1260--1281, 2019.

\bibitem{bredies2020sparsity}
K.~Bredies and M.~Carioni, ``Sparsity of solutions for variational inverse problems with finite-dimensional data,'' \emph{Calculus of Variations and Partial Differential Equations}, vol.~59, no.~1, p.~14, 2020.

\bibitem{fleming1957functions}
W.~H. Fleming, ``Functions with generalized gradient and generalized surfaces,'' \emph{Annali di Matematica Pura ed Applicata}, vol.~44, no.~1, pp. 93--103, 1957.

\bibitem{fleming1960functions}
------, ``Functions whose partial derivatives are measures,'' \emph{Illinois Journal of Mathematics}, vol.~4, no.~3, pp. 452--478, 1960.

\bibitem{ambrosio2001connected}
L.~Ambrosio, V.~Caselles, S.~Masnou, and J.-M. Morel, ``Connected components of sets of finite perimeter and applications to image processing,'' \emph{Journal of the European Mathematical Society}, vol.~3, no.~1, pp. 39--92, 2001.

\bibitem{unser1999splines}
M.~Unser, ``Splines: A perfect fit for signal and image processing,'' \emph{IEEE Signal Processing Magazine}, vol.~16, no.~6, pp. 22--38, 1999.

\bibitem{debarre2019b}
T.~Debarre, J.~Fageot, H.~Gupta, and M.~Unser, ``B-spline-based exact discretization of continuous-domain inverse problems with generalized {TV} regularization,'' \emph{IEEE Transactions on Information Theory}, vol.~65, no.~7, pp. 4457--4470, 2019.

\bibitem{debarre2022uniqueness}
T.~Debarre, Q.~Denoyelle, and J.~Fageot, ``On the uniqueness of solutions for the basis pursuit in the continuum,'' \emph{Inverse Problems}, vol.~38, no.~12, p. 125005, 2022.

\bibitem{guillemet2025convergence}
V.~Guillemet, J.~Fageot, and M.~Unser, ``Convergence analysis of the discretization of continuous-domain inverse problems,'' \emph{Inverse Problems}, vol.~41, no.~4, p. 045008, 2025.

\bibitem{fageot2025variational}
J.~Fageot, ``Variational seasonal-trend decomposition with sparse continuous-domain regularization,'' \emph{arXiv preprint arXiv:2505.10486}, 2025.

\bibitem{guillemet2025tensor}
V.~Guillemet and M.~Unser, ``Variational tensor product splines,'' \emph{En Preparation}, 2025.

\bibitem{bredies2010total}
K.~Bredies, K.~Kunisch, and T.~Pock, ``Total generalized variation,'' \emph{SIAM Journal on Imaging Sciences}, vol.~3, no.~3, pp. 492--526, 2010.

\bibitem{bredies2014regularization}
K.~Bredies and M.~Holler, ``Regularization of linear inverse problems with total generalized variation,'' \emph{Journal of Inverse and Ill-Posed Problems}, vol.~22, no.~6, pp. 871--913, 2014.

\bibitem{demengel1984fonctions}
F.~Demengel, ``Fonctions {\`a} hessien born{\'e},'' in \emph{Annales de l'institut Fourier}, vol.~34, no.~2, 1984, pp. 155--190.

\bibitem{lefkimmiatis2011hessian}
S.~Lefkimmiatis, A.~Bourquard, and M.~Unser, ``Hessian-based norm regularization for image restoration with biomedical applications,'' \emph{IEEE Transactions on Image Processing}, vol.~21, no.~3, pp. 983--995, 2011.

\bibitem{lefkimmiatis2013hessian}
S.~Lefkimmiatis, J.~P. Ward, and M.~Unser, ``{Hessian Schatten}-norm regularization for linear inverse problems,'' \emph{IEEE Transactions on Image Processing}, vol.~22, no.~5, pp. 1873--1888, 2013.

\bibitem{aziznejad2023measuring}
S.~Aziznejad, J.~Campos, and M.~Unser, ``Measuring complexity of learning schemes using {Hessian-Schatten} total variation,'' \emph{SIAM Journal on Mathematics of Data Science}, vol.~5, no.~2, pp. 422--445, 2023.

\bibitem{9655475}
J.~Campos, S.~Aziznejad, and M.~Unser, ``Learning of continuous and piecewise-linear functions with {Hessian} total-variation regularization,'' \emph{IEEE Open Journal of Signal Processing}, vol.~3, pp. 36--48, 2022.

\bibitem{pourya2023delaunay}
M.~Pourya, A.~Goujon, and M.~Unser, ``Delaunay-triangulation-based learning with {Hessian} total-variation regularization,'' \emph{IEEE Open Journal of Signal Processing}, vol.~4, pp. 167--178, 2023.

\bibitem{pourya2024box}
M.~Pourya, A.~Boquet-Pujadas, and M.~Unser, ``A box-spline framework for inverse problems with continuous-domain sparsity constraints,'' \emph{IEEE Transactions on Computational Imaging}, 2024.

\bibitem{ambrosio2023functions}
L.~Ambrosio, C.~Brena, and S.~Conti, ``Functions with bounded {Hessian--Schatten} variation: Density, variational, and extremality properties,'' \emph{Archive for Rational Mechanics and Analysis}, vol. 247, no.~6, p. 111, 2023.

\bibitem{ambrosio2024linear}
L.~Ambrosio, S.~Aziznejad, C.~Brena, and M.~Unser, ``Linear inverse problems with {Hessian--Schatten} total variation,'' \emph{Calculus of Variations and Partial Differential Equations}, vol.~63, no.~1, p.~9, 2024.

\bibitem{unser2022convex}
M.~Unser and S.~Aziznejad, ``Convex optimization in sums of {Banach} spaces,'' \emph{Applied and Computational Harmonic Analysis}, vol.~56, pp. 1--25, 2022.

\bibitem{dubins1962extreme}
L.~E. Dubins, ``On extreme points of convex sets,'' \emph{Journal of Mathematical Analysis and Applications}, vol.~5, no.~2, pp. 237--244, 1962.

\bibitem{parhi2023sparsity}
R.~Parhi and M.~Unser, ``The sparsity of cycle spinning for wavelet-based solutions of linear inverse problems,'' \emph{IEEE Signal Processing Letters}, vol.~30, pp. 568--572, 2023.

\bibitem{unser2025universal}
M.~Unser and S.~Ducotterd, ``Universal architectures for the learning of polyhedral norms and convex regularizers,'' \emph{arXiv preprint arXiv:2503.19190}, 2025.

\bibitem{bellettini2002total}
G.~Bellettini, V.~Caselles, and M.~Novaga, ``The total variation flow in rn,'' \emph{Journal of Differential Equations}, vol. 184, no.~2, pp. 475--525, 2002.

\bibitem{andreu2001minimizing}
F.~Andreu, C.~Ballester, V.~Caselles, and J.~M. Maz{\'o}n, ``Minimizing total variation flow,'' 2001.

\bibitem{andreu2002some}
F.~Andreu, V.~Caselles, J.~I. D{\'\i}az, and J.~M. Maz{\'o}n, ``Some qualitative properties for the total variation flow,'' \emph{Journal of Functional Analysis}, vol. 188, no.~2, pp. 516--547, 2002.

\bibitem{caselles2011regularity}
V.~Caselles, A.~Chambolle, and M.~Novaga, ``Regularity for solutions of the total variation denoising problem,'' \emph{Revista Matem{\'a}tica Iberoamericana}, vol.~27, no.~1, pp. 233--252, 2011.

\bibitem{lai2009convergence}
M.-J. Lai, B.~Lucier, and J.~Wang, ``The convergence of a central-difference discretization of {Rudin-Osher-Fatemi} model for image denoising,'' in \emph{Scale Space and Variational Methods in Computer Vision: Second International Conference, SSVM 2009, Voss, Norway, June 1-5, 2009. Proceedings 2}.\hskip 1em plus 0.5em minus 0.4em\relax Springer, 2009, pp. 514--526.

\bibitem{bartels2012total}
S.~Bartels, ``Total variation minimization with finite elements: convergence and iterative solution,'' \emph{SIAM Journal on Numerical Analysis}, vol.~50, no.~3, pp. 1162--1180, 2012.

\bibitem{wang2011error}
J.~Wang and B.~J. Lucier, ``Error bounds for finite-difference methods for {Rudin--Osher--Fatemi} image smoothing,'' \emph{SIAM Journal on Numerical Analysis}, vol.~49, no.~2, pp. 845--868, 2011.

\bibitem{chambolle2021approximating}
A.~Chambolle and T.~Pock, ``Approximating the total variation with finite differences or finite elements,'' in \emph{Handbook of Numerical Analysis}.\hskip 1em plus 0.5em minus 0.4em\relax Elsevier, 2021, vol.~22, pp. 383--417.

\bibitem{cohn2013measure}
D.~L. Cohn, \emph{Measure Theory}.\hskip 1em plus 0.5em minus 0.4em\relax Springer, 2013, vol.~2.

\bibitem{kallenberg2017random}
O.~Kallenberg \emph{et~al.}, \emph{Random Measures, Theory and Applications}.\hskip 1em plus 0.5em minus 0.4em\relax Springer, 2017, vol.~1.

\bibitem{federer1959curvature}
H.~Federer, ``Curvature measures,'' \emph{Transactions of the American Mathematical Society}, vol.~93, no.~3, pp. 418--491, 1959.

\bibitem{fleming1960integral}
W.~H. Fleming and R.~Rishel, ``An integral formula for total gradient variation,'' \emph{Archiv der Mathematik}, vol.~11, pp. 218--222, 1960.

\bibitem{rotem2023anisotropic}
L.~Rotem, ``The anisotropic total variation and surface area measures,'' in \emph{Geometric Aspects of Functional Analysis: Israel Seminar (GAFA) 2020-2022}.\hskip 1em plus 0.5em minus 0.4em\relax Springer, 2023, pp. 297--312.

\bibitem{dolzmann1995microstructures}
G.~Dolzmann and S.~M{\"u}ller, ``Microstructures with finite surface energy: the two-well problem,'' \emph{Archive for Rational Mechanics and Analysis}, vol. 132, pp. 101--141, 1995.

\bibitem{bonettini2018inertial}
S.~Bonettini, S.~Rebegoldi, and V.~Ruggiero, ``Inertial variable metric techniques for the inexact forward--backward algorithm,'' \emph{SIAM Journal on Scientific Computing}, vol.~40, no.~5, pp. A3180--A3210, 2018.

\bibitem{goujon2023neural}
A.~Goujon, S.~Neumayer, P.~Bohra, S.~Ducotterd, and M.~Unser, ``A neural-network-based convex regularizer for inverse problems,'' \emph{IEEE Transactions on Computational Imaging}, vol.~9, pp. 781--795, 2023.

\end{thebibliography}

\appendix
\subsection{Proofs }

\label{app:proofs}

\begin{proof}[\textbf{Proof of Lemma \ref{lemma:approximation}}]
\hfill\\\textbf{Construction.}
Consider the partition $(A_k^n)_{k=1}^n$ of $[0,1]$ with 
\begin{equation}
    A_{k}^n = \begin{cases}
    \left]\frac{k-1}{n}, \frac{k}{n}\right], & k > 1, \\
    \left[\frac{k-1}{n}, \frac{k}{n}\right], & k = 1,
    \end{cases}
\end{equation}
and the partition $(A_{k,k'}^{n})_{k,k'=1}^{n,n}$ of $\mathrm{K}$ with $A_{k,k'}^{n} = A_k^n \times A_{k'}^n.$ We define $m = [\DD]\{f\}$ and  
\begin{equation}
    f_n = \sum_{k=1}^n \sum_{k'=1}^n m(A_{k,k'}^{n}) \, u\left(\cdot - \frac{k}{n}\right) \otimes u\left(\cdot - \frac{k'}{n}\right).
\end{equation}
We calculate that 
\begin{equation}
    m_n = [\DD]\{f_n\} = \sum_{k=1}^n \sum_{k'=1}^n m(A_{k,k'}^{n}) \, \delta_{\frac{k}{n}} \otimes \delta_{\frac{k'}{n}} \in \M(\K),
\end{equation}
and observe that, by construction, $f_{n}$ is compactly supported in $\K$. Therefore, $f_n \in \M_{\DD,0}(\K)$.
\hfill\\\textbf{Item 1.}
For $[\DD]\{\phi\} \in [\DD]\{\C_0(\R^2)\}$, we calculate that
\begin{align}
    &\vert\langle f-f_n,[\DD]\{\phi\}\rangle\vert\nonumber\\ 
    =& \vert\langle m-m_n,\phi\rangle\vert\nonumber\\
    \leq\;& \sum_{k,k'=1,1}^{n,n} \left\vert \int_{A_{k,k'}^{n}} \phi(\bb{x}) \,\mathrm{d}m(\bb{x}) - m(A_n^k) \phi\left(\frac{k}{n}, \frac{k'}{n}\right) \right\vert\nonumber\\
    \leq\;& \sum_{k',k=1}^{n,n} \int_{A_{k,k'}^n} \left\vert \phi(\bb{x}) - \phi\left(\frac{k}{n}, \frac{k'}{n}\right) \right\vert \mathrm{d}m(\bb{x})\nonumber\\
    \leq\;& \norm{m}_{\M} \underset{\norm{\bb{x}-\bb{t}} \leq \frac{\sqrt{2}}{n}}{\text{sup}} \vert \phi(\bb{x}) - \phi(\bb{t})\vert \underset{n \to \infty}{\to} 0 \label{eq:2.B.42},
\end{align}
because $\phi$ is uniformly continuous. This shows the $\w$-convergence. 
\hfill\\\textbf{Item 2.} The observation that $\norm{m_n}_{\M} \leq \norm{m}_{\M}$ and the $\w$-convergence proved in Item 1 imply that 
\begin{equation}
    \norm{m}_{\M} \leq \underset{n \to \infty}{\text{liminf}} \norm{m_n}_{\M} \leq \underset{n \to \infty}{\text{limsup}} \norm{m_n}_{\M} \leq \norm{m}_{\M}.
\end{equation}
\textbf{Item 3.} We calculate that 
\begin{align}
f_n(x_1,x_2) = \int_{0}^{x_1} \int_{0}^{x_2} \mathrm{d}m_n(t_1,t_2) \leq \norm{m_n}_{\M} \leq \norm{m}_{\M}.
\end{align}
It follows that the family $\{(f_n)_{n=1}^{\infty}, f\}$ is upper-bounded by $\norm{m}_{\M} \mathbbm{1}_{\K}$. 
In addition, for $(x_1,x_2) \in A_{k,k'}^n$, we calculate that 
\begin{align}
    &f(x_1,x_2) - f_n(x_1,x_2)\nonumber\\
    =& m([0,x_1[ \times [0,x_2[) - m\left(\left[0,\frac{k-1}{n}\right] \times \left[0,\frac{k'-1}{n}\right]\right) \label{eq:A.181}\nonumber\\
    =& m\left(\left[0,\frac{k-1}{n}\right] \times \left]\frac{k'-1}{n}, x_2\right[\right) \nonumber \\
    & + m\left(\left]\frac{k-1}{n}, x_1\right[ \times \left[0, \frac{k'-1}{n}\right]\right) \nonumber \\
    & + m\left(\left]x_1, \frac{k-1}{n}\right[ \times \left]x_2, \frac{k'-1}{n}\right[\right) \underset{n \to \infty}{\to} 0.
\end{align}
In \eqref{eq:A.181}, we used the observation that taking the measure $m$ on the closed square or the half-open square yields functions that belong to the same equivalence class. The convergence follows from the continuity from above of measures. Consequently, since the sequence $(f_n)_{n=1}^{\infty}$ is upper-bounded by an $\mathcal{L}_1(\K)$ function and converges pointwise almost everywhere to $f$, we conclude by the Lebesgue dominated convergence theorem (LDC) that it converges in $\mathcal{L}_1(\K)$ to $f$.
\hfill\\\textbf{Item 4:} We first recall that, for $\K\subset\Omega$,
\begin{equation}
    \norm{\nabla\{f_n\}}_{\M^2}=\norm{\nabla\{f_n\}}_{\M(\K)^2}=\norm{\nabla\{f_n\}}_{\M(\Omega)^2}.
\end{equation}
Since $f_n \in \mathrm{BV}(\Omega)$, one has from the theory of BV functions that $\norm{\nabla\{f_n\}}_{\M(\Omega)^2}$ is equal to 
\begin{align} 
\text{sup} \left\{ \int_{\Omega} \nabla\{f_n\}(\bb{x}) \cdot \phi(\bb{x})\,\mathrm{d}\bb{x} : \phi \in \C_c^1(\Omega, \R^2), \norm{\phi} \leq 1 \right\} \label{eq:2.B.52}.
\end{align}
For $\phi \in \C_c^1(\Omega)$, we calculate that 
\begin{align}
    \langle [\mathrm{I} \otimes \Dd] \{f - f_n\}, \phi \rangle &= \langle (u \otimes \delta_0) \ast (m - m_n), \phi \rangle\nonumber \\
    &= \langle m - m_n, (u^{\vee} \otimes \delta_0) \ast \phi \rangle \underset{n \to \infty}{\to} 0,
\end{align}
because, since $m - m_n$ is compactly supported in $\K$, we can use the same argument as in \eqref{eq:2.B.42} with 
$(u^{\vee} \otimes \delta_0) \ast \phi \in \C(\R^2)$. It follows that, $\forall \phi \in \C_c^1(\Omega, \R^2)$,
\begin{equation}
\label{eq:conv.78}
    \langle \nabla\{f_n - f\}, \phi \rangle \underset{n \to \infty}{\to} 0 \Rightarrow \norm{\nabla\{f\}}_{\M^2} \leq \underset{n \to \infty}{\mathrm{liminf}} \norm{\nabla\{f_n\}}_{\M^2}.
\end{equation}
To continue,  we observe that 
\begin{multline}
\label{eq:conv.79}
    \norm{[\mathrm{I}\otimes\Dd]\{f_n\}}_{\M}=\norm{(u\otimes\delta_0)\ast m_n}_{\M}=\\
    \int_{0}^1\norm{m_n([0,t[\times\{\cdot\})}_{\M(\R)}\,\mathrm{d}t.
\end{multline}
Equation \eqref{eq:conv.79} also holds with $f_n$ replaced by $f$, and $m_n$ by $m$. If $t\in A_k^{n},$ then
\begin{align}
&\norm{m_n([0,t[\times\{\cdot\})}_{\M(\R)}\nonumber\\
    =&\sum_{k'=1}^{n}\left\vert m_n\left([0,t[\times\left\{\frac{k'}{n}\right\}\right)\right\vert\nonumber\\
    =&\sum_{k'=1}^{n}\left\vert\sum_{\ell=1}^{k-1}m(A_{\ell,k'}^n)\right\vert\nonumber\\
    =&\sum_{k'=1}^{n}\left\vert m\left(\left[0,\frac{k-1}{n}\right]\times A_{k'}^n\right)\right\vert\nonumber\\
    \leq&\sum_{k'=1}^{n}\left\vert m\left(\left[0,t\right[\times A_{k'}^n\right)\right\vert
    +\sum_{k'=1}^{n}\left\vert m\left(\left]\frac{k-1}{n},t\right[\times A_{k'}^n\right)\right\vert\nonumber\\
    \leq&\norm{m([0,t[\times\{\cdot\})}_{\M(\R)}+\norm{m\left(\left]\frac{k-1}{n},t\right[\times\{\cdot\}\right)}_{\M(\R)}.
\end{align}
It follows that 
\begin{multline}
\label{eq:conv.85}
    \norm{[\mathrm{I}\otimes\Dd]\{f_n\}}_{\M}\leq\norm{[\mathrm{I}\otimes\Dd]\{f\}}_{\M}\\+\int_{0}^1\norm{m(]\text{proj}_n(t),t[\times\{\cdot\})}_{\M(\R)}\,\mathrm{d}t,
\end{multline}
where $\text{proj}_{n}(t)$ returns the largest $(k-1)$ such that $\frac{k-1}{n}\leq t$. We write $T_n=]\text{proj}_n(t),t[$ and observe that a similar argumentation will show that 
\begin{multline}
\label{eq:conv.89}
    \norm{[\Dd\otimes\mathrm{I}]\{f_n\}}_{\M}\leq\norm{[\Dd\otimes\mathrm{I}\{f\}}_{\M}\\+\int_{0}^1\norm{m(\{\cdot\}\times T_n)}_{\M(\R)}\,\mathrm{d}t.
\end{multline}
Next, the combination of \eqref{eq:conv.78}, \eqref{eq:conv.85}, and \eqref{eq:conv.89} yields that 
\begin{equation}
\label{eq:+C}
    \norm{\nabla\{f\}}_{\M^2}
    \leq\underset{n\to\infty}{\mathrm{ limsup}}\norm{\nabla\{f_n\}}_{\M^2}
    \leq\norm{\nabla\{f\}}_{\M^2}+C
\end{equation}
with 
\begin{multline}
    C=\underset{n\to\infty}{\text{limsup }}\int_{0}^1\Big(\norm{m(T_n\times\{\cdot\})}_{\M(\R)}\\+\norm{m(\{\cdot\}\times T_n)}_{\M(\R)}\Big)\,\mathrm{d}t.
\end{multline}
Finally, we observe that 
\begin{equation}
\label{eq:conv.86}
    \Big(\norm{m(T_n \times \{\cdot\})}_{\M(\R)} + \norm{m(\{\cdot\} \times T_n)}_{\M(\R)} \Big) \leq 2 \norm{m}_{\M},
\end{equation}
and that 
\begin{equation}
\label{eq:conv.87}
    \underset{n \to \infty}{\text{lim}} \Big( \norm{m(T_n \times \{\cdot\})}_{\M(\R)} + \norm{m(\{\cdot\} \times T_n)}_{\M(\R)} \Big) = 0.
\end{equation}
because
\begin{equation}
    \bigcap_{n}T_n=\emptyset.
\end{equation}
Equation \eqref{eq:conv.87} follows from the continuity from above of the total variation measure. It follows from \eqref{eq:conv.86} and \eqref{eq:conv.87} that one can apply the LDC to find that $C = 0$. The combination of Equations \eqref{eq:conv.78} and \eqref{eq:+C} imply the desired equality.
\end{proof}

\end{document}